\documentclass[11pt,reqno,twoside]{article}
\usepackage{amssymb,amsfonts,amsthm,amsmath}
\usepackage{color}

\usepackage{hyperref}
\usepackage{cite}
\usepackage[left=1 in,top=1 in,right=1 in,bottom=1 in]{geometry}

\numberwithin{equation}{section}

\def\eqdef{\stackrel{\rm def}{=}}

\def\beq{\begin{equation}}
\def\eeq{\end{equation}}
\def\beqs{\begin{equation*}}
\def\eeqs{\end{equation*}}

\newtheorem{theorem}{Theorem}[section]
\newtheorem{lemma}[theorem]{Lemma}
\newtheorem{proposition}[theorem]{Proposition}
\newtheorem{corollary}[theorem]{Corollary}

\theoremstyle{definition}
\newtheorem{remark}[theorem]{Remark}

\newtheorem{definition}[theorem]{Definition}

\def\myclearpage{}

\definecolor{darkred}{rgb}{.70,.12,.20}

\definecolor{darkgreen}{rgb}{.20,.52,.14}

\definecolor{byz}{rgb}{.44,.16,.39}

\newcommand{\varep}{\varepsilon}
\newcommand{\norm}[1]{\left\| #1 \right\|  }

\newcommand{\coefset}{\mathcal R}
\newcommand{\indic}{{\mathbf 1}}

\numberwithin{equation}{section}

\title{Fluid Flows of Mixed Regimes in Porous Media}
\author{Emine Celik$^a$, Luan Hoang$^a$, Akif Ibragimov$^a$ and Thinh Kieu$^b$}

\date{\today}

\begin{document}
\maketitle
\begin{center}
\textit{$^a$Department of Mathematics and Statistics, Texas Tech University, Box 41042\\ Lubbock, TX 79409--1042, U. S. A.} \\
\textit{$^b$Department of Mathematics, University of North Georgia, Gainesville Campus\\ 3820 Mundy Mill Rd., Oakwood, GA 30566, U. S. A.}\\
Email addresses:  \texttt{emine.celik@ttu.edu, luan.hoang@ttu.edu,\\  akif.ibraguimov@ttu.edu, thinh.kieu@ung.edu}
\end{center}
\begin{abstract}
In porous media, there are three known regimes of fluid flows, namely, pre-Darcy, Darcy and post-Darcy. Because of their different natures, these are usually  treated separately in literature. To study complex flows when all three regimes may be present in different portions of a same domain, we use a single equation of motion to unify them. Several scenarios and models are then considered for  slightly compressible fluids. A nonlinear  parabolic equation for the pressure is derived, which is degenerate when the pressure gradient is either small or large. We estimate the pressure and its gradient for all time in terms of initial and boundary data. We also obtain their particular bounds for large time which depend on the asymptotic behavior of the boundary data but not on the initial one. Moreover, the continuous dependence of the solutions on  initial and boundary data, and the structural stability for the equation are established.
\end{abstract}


\pagestyle{myheadings}\markboth{E. Celik, L. Hoang, A. Ibragimov and T. Kieu}
{Fluid Flows of Mixed Regimes in Porous Media}

\myclearpage
\section{Introduction and the models}\label{intro}

Fluid flows are very common in nature such as in soil, sand, aquifers, oil reservoir, sea ice, plants, bones, etc.
Contrary to the usual perception of their simplicity, they, in fact, can be very complicated and are modeled by many different equations of various types.
Broadly speaking, they are categorized into three known regimes, namely, pre-Darcy (i.e. pre-linear, non-Darcy),  Darcy (linear) and post-Darcy (i.e. post-linear, non-Darcy).
While the Darcy regime is well-known, the other two do exist and are studied in physics and engineering. For example, when the Reynolds number is high, there is a deviation from the Darcy law and Forchheimer's equations are usually used to account for it \cite{Forchh1901,ForchheimerBook}, see also \cite{Muskatbook,BearBook,NieldBook}. 
On the other end of the Reynolds number's range, when it is small,  the pre-Darcy regime is observed but not well understood, although it contributes to unexpected oil extraction, see \cite{Dudgeon85,SSHI2016,SoniIslamBasak78} and references therein.

Concerning mathematical research of fluids in porous media, the flows' diverse nature   is much overlooked.
Almost  all of the papers  focus on the Darcy regime which is presented by the (linear) Darcy equation, see e.g. \cite{VazquezPorousBook}.
The post-Darcy regime has been attracted attention recently with the (nonlinear) Forchheimer models, see \cite{StraughanBook,ABHI1,HI2,HIK1,HIK2,HKP1,HK2,CHK1,CHK2} and references therein.
In contrast, the (nonlinear) pre-Darcy regime is virtually ignored.
Moreover,  the three regimes are always treated separately.
This is due to the different natures of the models and the ranges of their applicability.
However, this separation  is unsatisfactory since the fluid may present all three regimes in different unidentified portions of the confinement. 
Therefore, there is a need to unify the three regimes into one formulation and study the fluid as a whole. This paper aims at deriving admissible models for this unification and analyze their properties mathematically. 
To the best of our knowledge, this is the first paper to treat such a problem with rigorous mathematics.

We now start the investigation of different types of fluid flows in porous media.
Consider fluid flows with velocity $v\in \mathbb R^n$, pressure $p\in \mathbb R$, and density $\rho\in[0,\infty)$.
Depending on the range of the Reynolds number, there are different groups of  equations to describe their dynamics.

The most popular equation is Darcy's law:
\beq\label{D}
v=- k\nabla p, \text{ where $k$ is a positive constant.} 
\eeq
(In this paper, we will not discuss other  variations, such as those of Brinkman-type, for \eqref{D} or Forchheimer equations \eqref{F2}--\eqref{FP}.)

When $|v|$ is small, there are Izbash-type  equations  that describe the pre-Darcy regime:
\beq\label{PD}
|v|^{-\alpha}v=- k \nabla p\quad\text{for some constant power } \alpha\in (0,1)\text{ and  coefficient }k>0.
\eeq
For experimental values of  $\alpha$, see e.g.  \cite{SSHI2016,SoniIslamBasak78}.

When $|v|$ is large, the following Forchheimer equations are usually used in studying post-Darcy flows.

Forchheimer's two-term law
\beq\label{F2}
av+b|v|v=-\nabla p.
\eeq

Forchheimer's three-term law
\beq\label{F3}
av+b|v|v+c|v|^2v=-\nabla p. 
\eeq

Forchheimer's power law
\beq\label{FP}
av+d|v|^{m-1}v=-\nabla p.
\eeq
Here, the positive numbers $a$, $b$, $c$, $d$, and $m\in(1,2)$ are derived from experiments for each case.

The above three Forchheimer equations can be combined and generalized to the following form:
\beq\label{gF}
g_F(|v|)v=-\nabla p,
\eeq
where
\beq
g_F(s)=a_0+a_1s^{\alpha_1}+\dots+a_Ns^{\alpha_N},
\eeq
with $N\ge 1$, $\alpha_0=0<\alpha_1<\ldots<\alpha_N$, $a_0,a_N>0$, $a_1,a_2,\ldots,a_{N-1}\ge 0$.

The generalized Forchheimer equation \eqref{gF} was intensely used by the authors to model and study fast flows in the porous media (see \cite{ABHI1,HI1,HI2,HIKS1,HKP1,HK1,HK2,CH1,CH2,CHK1,CHK2}). The techniques developed in those papers will be essential in our approach and analysis below.

In previous work, each regime pre-Darcy, Darcy, or post-Darcy was studied separately, even though they  exist simultaneously in porous media. In particular cases, some models must consider multi-layer domains with each layer having a different regime of fluid flows, see for e.g. section 6.7.8 of \cite{StraughanBook}.
The goal of this section is to model all regimes together in the same domain. 

\newcommand{\Gfield}{\mathbf G}

We write  a general equation of motion for all cases \eqref{D}--\eqref{gF} as
\beq\label{vecform}
\Gfield(v)=-\nabla p,
\eeq
where $\Gfield$ is a vector field on $\mathbb R^n$ with $\Gfield(0)=0$.
In this paper, based on the known equations \eqref{PD}--\eqref{gF},  we study $\Gfield$ of the form
 \beq\label{gg}
  \Gfield(v)=
\begin{cases}
g(|v|)v & \text{if } v\in \mathbb R^n\setminus \{0\},\\
0& \text{if } v=0,
\end{cases}
 \eeq
where $g(s)$ is a continuous function from $(0,\infty)$ to $(0,\infty)$ that satisfies
\beq\label{sgs}
\lim_{s\searrow 0} sg(s)=0.
\eeq

Different forms of $g(s)$ give different models, for example, 
\beqs
g(s)=k^{-1}s^{-\alpha},\
k^{-1},\
a+bs,\
a+bs+cs^2,\
a+ds^{m-1},\
g_F(s),
\eeqs
for equations \eqref{PD}, \eqref{D}, \eqref{F2}, \eqref{F3}, \eqref{FP}, \eqref{gF}, respectively.

As in our previous work for compressible fluids, to reduce the complexity of the system of equations describing the fluid motion, we solve for velocity $v$ in \eqref{gg} in terms of the pressure gradient $\nabla p$.
For example, from  \eqref{PD} we have
\beqs
v=-k^\frac{1}{1-\alpha}|\nabla p|^\frac{\alpha}{1-\alpha}\nabla p,
\eeqs
and from  \eqref{gF} we have 
\beqs
v=-K_F(|\nabla p|)\nabla p, 
\eeqs
where 
\beq\label{KF}
K_F(\xi)=\frac{1}{g_F(G_F^{-1}(\xi))}  \quad  \text{with }G_F(s)=sg_F(s) \quad\text{for }\xi, s\ge 0.
\eeq

Taking the modulus both sides of \eqref{vecform}, we have
\beq\label{vG}
G(|v|)=|\nabla p|,
\eeq
where
 \beq\label{Gdef}
  G(s)=
\begin{cases}
sg(s)& \text{if } s>0,\\
0& \text{if } s=0.
\end{cases}
 \eeq

By \eqref{sgs}, we have 
\begin{enumerate}
\item[{\rm (g1)}] $G(s)$ is continuous on $[0,\infty)$. 
\end{enumerate}

We assume
\begin{enumerate}
\item[{\rm (g2)}] $G(s)$ is strictly increasing on $[0,\infty)$,
\item[{\rm (g3)}] $G(s)\to\infty$ as $s\to\infty$, and
\item[{\rm (g4)}]  the function $1/g(s)$ on $(0,\infty)$  can be extended to a continuous function $k_g(s)$ on $[0,\infty)$.
\end{enumerate}

By (g1)--(g3), we can invert equation \eqref{vG} to have 
$$|v|=G^{-1}(|\nabla p|).$$

Combining this with (g4), we can solve from \eqref{vecform} and \eqref{gg} for
$v=- k_g(|v|)\nabla p$,
thus,
\begin{equation}\label{gen-darcy}
v=-K(|\nabla p|)\nabla p,
\end{equation}
where
\beq\label{Kg} 
K(\xi)= k_g(G^{-1}(\xi))\quad\text{for } \xi\ge 0.
\eeq 
In particular, when $\xi>0$
\beq\label{Kpos}
K(\xi)=\frac{1}{g(s(\xi))} \quad \text{with}\quad  s=s(\xi)>0 \quad \text{satisfying} \quad s g(s)=\xi.
\eeq

One can interpret equation \eqref{gen-darcy} as a generalization Darcy equation \eqref{D} with conductivity  $k=K(|\nabla p|)$ depending on  the pressure's gradient.

We consider the following two  main  models. Below, $\indic_E$ denotes the characteristic (indicator) function of a set $E$.

\textbf{Model 1.} Function $g(s)$ is piece-wise smooth on $(0, \infty)$. 
Based on \eqref{PD}, \eqref{D} and \eqref{gF} and their validity in different ranges  of $|v|$, our first consideration is the following piece-wise defined function
\beq \label{gdef}
g(s)=\bar{g}(s)\eqdef c_1 s^{-\alpha}\indic_{(0,s_1)}(s)+ c_2\indic_{[s_1,s_2]}(s)+g_F(s)\indic_{(s_2,\infty)}(s)
\quad \text{for } s>0,
\eeq
where $\alpha\in(0,1)$, and $s_2>s_1>0$ are fixed threshold values.
To avoid abrupt transitions between three regimes, we impose the continuity on $\bar g(s)$, that is,
\beqs
c_1s_1^{-\alpha}=c_2=g_F(s_2).
\eeqs

Note that $\bar g(s)$ is not differentiable at $s_1,s_2$.
Obviously, \eqref{sgs} holds.
Then function $G(s)$ in \eqref{Gdef} becomes
\beqs
G(s)=\bar G(s)\eqdef  c_1 s^{1-\alpha}\indic_{[0,s_1)}(s)+ c_2 s \indic_{[s_1,s_2]}(s)+G_F(s)\indic_{(s_2,\infty)}(s)
 \quad\text{for } s\ge 0.
\eeqs

Clearly, conditions (g2)--(g3) are satisfied. Then
\beqs
\bar G^{-1}(\xi)=\Big(\frac\xi{c_1}\Big)^\frac1{1-\alpha} \indic_{[0,Z_1)}(\xi) + \frac\xi{c_2} \indic_{[Z_1,Z_2]}(\xi)+G_F^{-1}(\xi)\indic_{(Z_2,\infty)}(\xi)\quad\text{for }  \xi\ge 0,
\eeqs
where $Z_1=c_2 s_1$ and  $Z_2=c_2 s_2$.
Also, (g4) holds true with 
\beqs
k_{\bar g}(s)=\frac{s^\alpha }{c_1} \indic_{[0,s_1)}(s) + \frac1{c_2} \indic_{[s_1,s_2]}(s)+\frac1{g_F(s)}\indic_{(s_2,\infty)}(s)\quad\text{for }  s\ge 0.
\eeqs
Thus, we derive function $K(\xi)$ in \eqref{Kg} explicitly as
\beq\label{K-bar}
K(\xi)=\bar K(\xi)\eqdef M_1\xi^{\beta_1} \indic_{[0,Z_1)}(\xi) + M_2 \indic_{[Z_1,Z_2]}(\xi)+K_F(\xi)\indic_{(Z_2,\infty)}(\xi)\quad\text{for }  \xi\ge 0,
\eeq
where  $M_1=c_1^{-\frac{1}{1-\alpha}}$, $M_2=c_2^{-1}$,   and
\beq\label{beta1}
\beta_1=\alpha/(1-\alpha)>0.
\eeq

Note that, similar to the function $\bar g$, this function $\bar K$ is continuous on $[0,\infty)$,  continuously differentiable on $(0,\infty)\setminus\{Z_1,Z_2\}$.


\textbf{Model 2.} Function $g(s)$ is smooth on $(0, \infty)$.
Another generalization is  to use a smooth interpolation between pre-Darcy \eqref{PD} and generalized Forchheimer \eqref{gF}. Instead of \eqref{gdef}, we propose the following
\begin{equation}\label{gI}
g(s)=g_I(s)\eqdef a_{-1}s^{-\alpha}+a_0+a_1s^{\alpha_1}+\dots+a_Ns^{\alpha_N} 
\quad \text{for } s>0,
\end{equation}
where  $N\ge 1$, $\alpha\in (0,1)$, $\alpha_N>0$, 
\beq\label{aicond} 
a_{-1},a_N>0 \text{ and }a_i\ge 0\quad  \forall i=0,1,\ldots, N-1.
\eeq

Normally, $a_0>0$ and, thus,  the model \eqref{gI} already contains the Darcy regime in its formulation.
Nonetheless, our mathematical study in this paper allows the case $a_0=0$ as well.

If only one function $g_I$ is studied, then we can impose $a_i> 0$ for all $i=-1,0,\ldots,N$.
 The weaker condition \eqref{aicond} is used here to allow comparison between $g_I$ functions with different powers $\alpha_i$, see section \ref{strucsec}.

The main advantage of $g_I$ over $\bar g$ is that it is smooth on $(0,\infty)$.
This  allows further mathematical analysis of the flows.
It also can be used as a framework for perspective interpretation of field data, i.e., matching the coefficients $a_i$ for $i=-1,0,1,\ldots,N$ to fit the data.   

Similar to Model 1, conditions \eqref{sgs} and (g2)--(g4) are satisfied with
\beqs
G(s)= G_I(s)\eqdef  a_{-1}s^{1-\alpha}+a_0 s +a_1s^{1+\alpha_1}+\dots+a_Ns^{1+\alpha_N}, 
\eeqs
\beqs
k_{g_I}(s)=\frac{s^\alpha}{a_{-1}+a_0 s^\alpha + a_1 s^{\alpha+\alpha_1}+\ldots+a_N s^{\alpha+\alpha_N}}
\eeqs
for $s\ge 0$. We then obtain
\beq\label{K-I}
K(\xi)=K_I(\xi)\eqdef \frac{s(\xi)^\alpha}{a_{-1}+a_0 s(\xi)^\alpha + a_1 s(\xi)^{\alpha+\alpha_1}+\ldots+a_N s(\xi)^{\alpha+\alpha_N}}\quad \text{for }  \xi\ge 0,
\eeq
where
$s(\xi)= G_I^{-1}(\xi)$.

In case we want to consider dependence on the coefficients of $g_I(s)$, we denote
\beq\label{KIxa}
\vec a=(a_{-1},a_0,a_1,\ldots,a_N),\quad g_I(s)=g_I(s,\vec a),\quad \text{and}\quad  K_I(\xi)=K_I(\xi,\vec a).
\eeq

\textbf{Model 3.}  Another  way to describe the flows of mixed regimes is to take the formula \eqref{gen-darcy} 
and define the conductivity function $K(\xi)$ directly that possesses  some desired properties. 
In doing so, one can impose the smoothness on $K(\xi)$.
An important feature in constructing $K(\xi)$ is to preserve its behavior when $\xi\to 0$ or $\xi\to \infty$ to be the  same as that of $\bar K(\xi)$.  
As $\xi\to 0$ it is clear from \eqref{K-bar} that $\bar K(\xi)$ is like $\xi^{\beta_1}$.
For sufficiently large $\xi$ we have $\bar K(\xi)=K_F(\xi)$ defined by  \eqref{KF}.
Thus, we recall from Lemma 2.1 of \cite{HI1}  that the function  $K_F(\xi)$  satisfies
\beq\label{Fdegen}
\frac{d_1^{-1}}{(1+\xi)^{\beta_2}}\le K_F(\xi)\le \frac{d_1}{(1+\xi)^{\beta_2}},
\eeq
where 
\beq \label{beta2}
\beta_2=\alpha_N/(1+\alpha_N)\in(0,1)\quad\text{and}\quad  d_1=d_0 (\max\{a_0,a_1,\ldots,a_N,a_0^{-1},a_N^{-1}\})^{1+\beta_2}
\eeq
with $d_0>0$ depending on $N$ and $\alpha_N$.

In summary, we want $K(\xi)$ to behave like $\xi^{\beta_1}$ for small $\xi$, and, as in \eqref{Fdegen}, like $(1+\xi)^{-\beta_2}$ for large $\xi$.
 Therefore, ones can introduce
\beq\label{Ksim1}
K(\xi)=\hat{K}(\xi)\eqdef \frac{a\xi^{\beta_1}}{(1+b\xi^{\beta_1})(1+c\xi^{\beta_2})}\quad\text{for }\xi\ge 0.
\eeq
Here, positive coefficients $a$, $b$, $ c$, and parameters $\beta _1$, $\beta_2$ can be used to match experimental or field data.
This function $\hat K$ belongs to $C^\infty((0,\infty))$.

\textbf{Model 4.} Ones can also refine the model \eqref{Ksim1} to match  more accurately  $\bar K(\xi)$ in \eqref{K-bar}.
Specifically,  $K(\xi)$  is close to $M_1 \xi^{\beta_1}$ when $\xi\to 0$, and  to $K_F(\xi)$ when $\xi\to \infty$.
Then we choose
\beq\label{KM}
K(\xi)=K_M(\xi)\eqdef K_F(\xi)\cdot\frac{\bar{k}\xi^{\beta_1}}{1+\bar{k}\xi^{\beta_1}} \quad\text{for }\xi\ge 0,
\eeq
where
$
\bar k=M_1/K_F(0)>0.
$

\vspace*{1em}
Above, we have introduced several models which can be used to interpret experimental and field data. We now use them to investigate the fluid flow's properties. They are used together with other basic equations of continuum mechanics which we recall here.

Continuity equation
\beqs
\phi\rho_t+\nabla\cdot(\rho v)=0,
\eeqs
where $\phi\in(0,1)$ is the constant porosity.

Constitutive law for slightly compressible fluids
\beqs
\frac{d\rho}{dp}=\frac{\rho}\kappa,
\eeqs
where $1/\kappa>0$ is small compressibility.

Combining the above two equations with \eqref{gen-darcy}, we obtain
\beqs
\phi p_t=\kappa \nabla \cdot(K|\nabla p|)\nabla p)+K(|\nabla p|)|\nabla p|^2.
\eeqs

Since $\kappa$ is large, we neglect the last term in this study. Such a simplification is commonly used in petroleum engineering. The full treatment requires more accurate models for the flows and can use the similar analysis as in \cite{CHK1,CHK2}.

By rescaling $t$, we assume $\kappa=1$ and obtain the following reduced equation
\beq\label{peq}
p_t=\nabla \cdot(K(|\nabla p|)\nabla p),
\eeq
where $K$ is $\bar K$, $K_I$, $\hat K$ or $K_M$.

We will study the initial, boundary value problem (IBVP) associated with the partial differential equation \eqref{peq}. 
We will derive estimates for the solutions, and establish their continuous dependence on the  initial and boundary data, and, in case $K=K_I$, on the  coefficients of the function $g_I(s)$ in \eqref{gI}. As seen in the next section, the PDE \eqref{peq} is degenerate when either $|\nabla p|\to 0$ or $|\nabla p|\to\infty$. Moreover, it possesses a monotonicity of mixed type which requires extra care in the proof and analysis.

 The paper is organized as follows. In section \ref{presec}, we present important properties of $K(\xi)$ including its type of degeneracy (Lemma \ref{lem21}) and  monotonicity (Lemma \ref{lemmono}).
They are essential not only for the remaining  sections \ref{boundsec}--\ref{strucsec} in this paper but also for our future work on the models.
In section \ref{boundsec}, we study solutions of \eqref{peq} subjected to the time-dependent Dirichlet boundary condition  $\psi(x,t)$. We derive  estimates   the $L^2$-norm  for a solution $p(x,t)$ and the $L^{2-\beta_2}$-norm  for  its gradient, both   for all $t\ge 0$ and, particularly, for large $t$, see Theorems~\ref{Lem31} and \ref{grad-est}. Furthermore,  we show in Theorems~\ref{pasmall}  and \ref{gradasmall}  that if the boundary data is asymptotically small as $t\to\infty$, then so are these two norms.  
Section \ref{dependsec} is focused on the continuous dependence of solutions on the initial and boundary data. Theorem~\ref{theo62} shows that the difference between two  solutions $p_1(x,t)$ and $p_2(x,t)$ with boundary data $\psi_1(x,t)$ and $\psi_2(x,t)$, respectively, is small if their initial difference $p_1(x,0)-p_2(x,0)$ and the boundary data's difference  $\psi_1(x,t)-\psi_2(x,t)$ are small, see \eqref{neww}. Especially when $t\to\infty$, the estimates of $p_1(x,t)-p_2(x,t)$ depend on the asymptotic behavior of  $\psi_1(x,t)-\psi_2(x,t)$.
In section \ref{strucsec}, we consider particularly $g=g_I(s,\vec a)$, $K=K_I(\xi,\vec a)$  and prove the structural stability of equation \eqref{peq} with respect to the coefficient vector $\vec a$ of the function $g_I$. In order to obtain this, we first establish in Lemma~\ref{lempm} the perturbed monotonicity for our degenerate PDE. It is then proved in Theorem~\ref{DepCoeff} that the difference $P(x,t)$ between   the two solutions which correspond to two different coefficient vectors $\vec a^{(1)}$ and  $\vec a^{(2)}$ is estimated in terms of their initial difference $P(x,0)$ and $|\vec a^{(1)}-\vec a^{(2)}|$, see \eqref{ssc2}. Moreover, when time goes to infinity, this difference can be  controlled by $|\vec a^{(1)}-\vec a^{(2)}|$ only, see \eqref{ssc3}.       

\section{Basic properties and inequalities}\label{presec}
In this section, we  study some properties of the conductivity function $K(\xi)$ which
 play crucial roles in the analysis of the PDE  \eqref{peq}.
For comparison purpose we define the function
\beq
K_*(\xi)=\frac{\xi^{\beta_1}}{(1+\xi)^{\beta_1+\beta_2}}\quad \text{for } \xi\ge 0,
\eeq
where $\beta_1>0$ and $\beta_2\in(0,1)$ are defined in \eqref{beta1} and \eqref{beta2}, respectively.

Let $\xi_c=\beta_1/\beta_2$.
It is elementary to see that 
\begin{enumerate}
 \item[{\rm (P1)}] $K_*(\xi)$ is increasing on $[0,\xi_c]$,
 \item[{\rm (P2)}] $K_*(\xi)$ is decreasing on $[\xi_c,\infty)$, and hence,
 \item[{\rm (P3)}]  $K_*(\xi_c)$ is the maximum of $K_*(\xi)$ over $[0,\infty)$.
\end{enumerate}

For $m,\xi\ge 0$,
 \beqs
 K_*(\xi)\xi^m=\Big(\frac{\xi}{1+\xi}\Big)^{\beta_1} \frac{\xi^m}{(1+\xi)^{\beta_2}}\le \frac{\xi^m}{(1+\xi)^{\beta_2}}.
 \eeqs
 Therefore,
  \beq\label{P4}
 K_*(\xi)\xi^m\le \xi^{m-\beta_2} \quad\forall m\ge 0,\ \xi\ge 0.
 \eeq
 
If $m\ge \beta_2$ and $\xi>\delta>0$ then 
  \beqs
     K_*(\xi)\xi^m= \Big(\frac{\xi}{1+\xi}\Big)^{\beta_1+\beta_2} \xi^{m-\beta_2}
      \ge \Big(\frac{\delta}{1+\delta}\Big)^{\beta_1+\beta_2} (\xi^{m-\beta_2}-\delta^{m-\beta_2}
     ). 
  \eeqs
This inequality is obviously true when $\xi\le \delta.$  Hence,
  \beq\label{P5}
     K_*(\xi)\xi^m \ge \Big(\frac{\delta}{1+\delta}\Big)^{\beta_1+\beta_2} (\xi^{m-\beta_2}-\delta^{m-\beta_2}
     )\quad \forall \delta>0, \ m\ge \beta_2,\  \xi\ge 0. 
  \eeq

\begin{lemma}\label{lem21}
Let $K=\bar K$, $K_I$, $\hat{K}$, and $K_M$ as in \eqref{K-bar}, \eqref{K-I},  \eqref{Ksim1} and \eqref{KM}, respectively.
Then there exist $d_2,d_3>0$ such that
\beq\label{mc1}
d_2K_*(\xi)\le K(\xi)\le d_3 K_*(\xi)\quad \forall \xi\ge 0.
\eeq
Consequently, for all $m\ge \beta_2$ and $\delta>0$,
\beq\label{mc2}
d_2\Big(\frac{\delta}{1+\delta}\Big)^{\beta_1+\beta_2} (\xi^{m-\beta_2}-\delta^{m-\beta_2}
     )\le K(\xi)\xi^m\le d_3 \xi^{m-\beta_2}\quad \forall \xi\ge 0.
\eeq

In particular, when $K=K_I$ ones can take
\beq\label{MM}
d_2=\frac1{\max\{1,\xi_0\}^{1+\beta_1}} \text{ and }
d_3=  \frac{(1+\max\{1,\beta_0\})^{\beta_1+\beta_2}}{\min\{1,a_{-1},a_N\}^{1+\beta_1}} 
\eeq
with $\xi_0=a_{-1}+a_0+a_1+\dots+a_N.$ 
\end{lemma}
\begin{proof}
The inequalities in \eqref{mc1} clearly holds for $K=\hat K$. 
When $K=K_M$, \eqref{mc1}  can be easily proved by using relation \eqref{Fdegen}.

We prove \eqref{mc1} for $K=\bar K$ now. 

For $0\le \xi<Z_1$, we have 
\beq\label{km1}
\frac{K(\xi)}{M_1(1+Z_1)^{\beta_1+\beta_2}}=\frac{\xi^{\beta_1}}{(1+Z_1)^{\beta_1+\beta_2}} \le K_*(\xi)\le \xi^{\beta_1}= \frac{K(\xi)}{M_1}.
\eeq

For $Z_1\le \xi\le Z_2$, we have 
\beq\label{km2}
\frac{Z_1^{\beta_1}K(\xi)}{M_2(1+Z_1)^{\beta_1}(1+Z_2)^{\beta_2}} 
=\Big(\frac{Z_1}{1+Z_1}\Big)^{\beta_1} \frac1{(1+Z_2)^{\beta_2}} \le K_*(\xi)\le  Z_2^{\beta_1}=\frac{Z_2^{\beta_1}K(\xi)}{M_2}.
\eeq
Above, we used, for the first inequality, the fact that the function $x/(x+1)$ is increasing.

For $\xi>Z_2$, we have from \eqref{Fdegen} that
\beq\label{km3}
\frac{Z_2^{\beta_1} }{(1+Z_2)^{\beta_1}} d_1^{-1}K(\xi) \le 
\frac{Z_2^{\beta_1}}{(1+Z_2)^{\beta_1}}\cdot \frac1{(1+\xi)^{\beta_2}} \le K_*(\xi)\le \frac{1}{(1+\xi)^{\beta_2}}\le d_1 K(\xi).
\eeq
Therefore, relation \eqref{mc1} follows \eqref{km1}, \eqref{km2}, and \eqref{km3}.

Next, consider $K=K_I$.  Since $K_I(0)=0$, it suffices to prove \eqref{mc1} for $\xi>0$.

Let $s=s(\xi)=G_I^{-1}(\xi)>0$. Then we have 
\beq \label{xis}
\xi=G_I(s)=a_{-1}s^{1-\alpha}+a_0s+a_1s^{1+\alpha_1}+\dots+a_Ns^{1+\alpha_N}.
\eeq
Note that $\xi_0=G_I(1)$. We consider the following two cases.  

\textit{Case 1: $\xi>\xi_0$.}  Then  $s>1$ and we have from \eqref{xis} that
 $$ a_Ns^{1+\alpha_N}\le \xi\le \xi_0s^{1+\alpha_N}.$$
This and the fact $ K_I(\xi)=1/g_I(s)=s/\xi$ give
\beq\label{ki1}
C_1 \xi^{-\beta_2}=\frac{(\xi/\xi_0)^\frac 1{1+\alpha_N}}{\xi} \le K_I(\xi)\le \frac{(\xi/a_N)^\frac 1{1+\alpha_N}}{\xi}
=C_2\xi^{-\beta_2},
\eeq
where 
$C_1=\xi_0^{\beta_2-1}$ and $C_2=a_N^{\beta_2-1}$.
The first inequality of \eqref{ki1} immediately yields the lower bound for $K_I(\xi)$ as
\beq\label{kxi40}
K_I(\xi)\ge \frac {C_1}{(1+\xi)^{\beta_2}} \ge   \frac {C_1}{(1+\xi)^{\beta_2}} \cdot \frac {\xi^{\beta_1}}{(1+\xi)^{\beta_1}}=C_1 K_*(\xi).
\eeq
For the upper bound of $K_I(\xi)$ we note  for $\xi> \xi_0$ that
\beq\label{ki2}
\frac{\xi_0}{\xi_0+1}\le \frac{\xi}{\xi+1} \text{ and } \xi\ge\frac{\xi+\xi_0}2\ge \min\{1,\xi_0\} \frac{\xi+1} 2.
\eeq 
Combining the second inequality of \eqref{ki1} and \eqref{ki2} gives 
\beq\label{ki4}
K_I(\xi)
\le C_2 \Big(\frac2{\xi+\xi_0}\Big)^{\beta_2}\Big(\frac{\xi_0+1}{\xi_0}\cdot \frac{\xi}{\xi+1}\Big)^{\beta_1}
\le C_3 \frac{\xi^{\beta_1}}{(\xi+1)^{\beta_1+\beta_2}}=C_3 K_*(\xi),
\eeq
where 
$$C_3=C_2\Big(\frac{\xi_0+1}{\xi_0}\Big)^{\beta_1} \Big(\frac2{\min\{1,\xi_0\}}\Big)^{\beta_2}.$$

\textit{Case 2: $0< \xi\le \xi_0$.} We have $0< s\le 1$ in this case and
\beq\label{ki6}
 \frac{s^\alpha}{\xi_0} \le K_I(\xi)=\frac{s^\alpha}{a_{-1}+a_0s^{\alpha_1+\alpha}+\dots+a_Ns^{\alpha_N+\alpha}}\le  \frac{s^\alpha}{a_{-1}}.
\eeq
By \eqref{xis}, 
\beqs 
a_{-1}s^{1-\alpha} \le \xi \le \xi_0s^{1-\alpha} \text{ which implies}\quad (\xi/\xi_0)^\frac1{1-\alpha}\le s\le (\xi/a_{-1})^\frac1{1-\alpha}.
\eeqs 

Utilizing this in \eqref{ki6} yields
\beqs
\frac {\xi^{\beta_1}}{\xi_0^{1+\beta_1}}\le K_I(\xi) \le \frac {\xi^{\beta_1}}{a_{-1}^{1+\beta_1}} .
\eeqs
Since $\xi\le \xi_0$, we obtain  
\beq\label{ki5}
C_4\frac {\xi^{\beta_1}}{(1+\xi)^{\beta_1+\beta_2}}  \le K_I(\xi) \le  \frac 1{a_{-1}^{1+\beta_1}} \xi^{\beta_1} \frac{(1+\xi_0)^{\beta_1+\beta_2}}{(1+\xi)^{\beta_1+\beta_2}}= C_5\frac {\xi^{\beta_1}}{(1+\xi)^{\beta_1+\beta_2}}  ,
\eeq
where $C_4=1/\xi_0^{1+\beta_1}$ and $C_5=(1+\xi_0)^{\beta_1+\beta_2}/a_{-1}^{1+\beta_1}$.
Combining inequalities \eqref{kxi40}, \eqref{ki4} and \eqref{ki5}, we have for both cases that 
\beq\label{ki3}
C_6K_*(\xi)  \le K_I(\xi) \le  C_7 K_*(\xi),
\eeq
where $C_6=\min\{C_1, C_4\}$ and $C_7=\max\{C_3, C_5\}$.

Note that 
$$
C_6=\frac1{\max\big\{\xi_0^{1+\beta_1},\xi_0^{1-\beta_2}\big\}} \ge \frac1{\max\{1,\xi_0\}^{1+\beta_1}}=d_2
$$
and
 \begin{align*}
C_7&= \max\Big\{\Big(\frac{\xi_0+1}{\xi_0}\Big)^{\beta_1} \frac{2^{\beta_2}}{a_N^{1-\beta_2} \min\{1,\xi_0\} ^{\beta_2}},\frac {(1+\xi_0)^{\beta_1+\beta_2}}{a_{-1}^{1+\beta_1}}\Big\} \\
&\le  \max\Big\{\frac{(\xi_0+1)^{\beta_1}  2^{\beta_2}}{(\min\{1,a_{-1},a_N\})^{\beta_1+(1-\beta_2)+\beta_2}},\frac {(1+\xi_0)^{\beta_1+\beta_2}}{(\min\{1,a_{-1},a_N\})^{1+\beta_1}}\Big\}\\ 
&\le  \frac{(1+\max\{1,\xi_0\})^{\beta_1+\beta_2}}{\min\{1,a_{-1},a_N\}^{1+\beta_1}}=d_3 .
\end{align*}
Hence we obtain \eqref{mc1} from \eqref{ki3}.

 Combining \eqref{mc1} with \eqref{P4} and \eqref{P5} gives \eqref{mc2}.
 The proof is complete.
\end{proof}

Combining \eqref{mc1} and  (P3) gives the upper bound for $K(\xi)$ as
\beq\label{Kbdd}
K(\xi)\le d_4\eqdef d_3K_*(\xi_c)\quad\forall \xi\in [0,\infty).
\eeq

\begin{lemma}\label{LemKp}
Let  $K=K_I$, $\hat{K}$, and $K_M$  as in  \eqref{K-I}, \eqref{Ksim1}, and  \eqref{KM}, respectively. Then for all $\xi>0$ one has 
\beq\label{Kderest}
 -\beta_2 \frac{K(\xi)}{\xi}\le K'(\xi) \le  \beta_1 \frac{K(\xi)}{\xi}
\eeq
and, consequently,
\beq\label{Kp2}
|K'(\xi)|\le \max\{\beta_1,\beta_2\}\frac{K(\xi)}{\xi}.
\eeq

In case $K=\bar K$ in \eqref{K-bar}, the inequalities \eqref{Kderest} and \eqref{Kp2} hold for $0<\xi\ne Z_1,Z_2$. 
\end{lemma}

\begin{proof}
Let $\xi>0$, then $ s(\xi)=G^{-1}(\xi)>0$.
First, consider $g=g_I$ and $K=K_I$.
By the chain rule, we have from \eqref{Kpos} that
\beq\label{Kp1}
K'(\xi)=-\frac{g'(s(\xi))s'(\xi)}{g(s(\xi))^2}.
\eeq
Denote $s=s(\xi)$, then $s\cdot g(s)=\xi$. Hence, $s'g(s)+sg'(s)s'=1$ which yields
\beq\label{s:prime}
   s'(\xi)=\frac{1}{g(s)+sg'(s)}.
\eeq
Substituting \eqref{s:prime} into \eqref{Kp1} yields 
\begin{align}
K'(\xi)&=-\frac{g'(s)}{g(s)^2}\cdot \frac{1}{g(s)+sg'(s)}
=-\frac{1}{g(s)}\cdot \frac{1}{sg(s)} \cdot \frac{g'(s)s}{g(s)+s g'(s)} \nonumber\\
&=-\frac{K(\xi)}{\xi}\cdot \frac{g'(s)s}{ g(s)+ s g'(s)}=\frac{K(\xi)}{\xi}\Big(\frac{g(s)}{ g(s)+ s g'(s)}-1\Big) \label{Kprime}.
\end{align}
 
Now, for any $s>0$,  ones observe that
\beqs
g(s)+ s g'(s)= (1-\alpha) a_{-1}s^{-\alpha}+a_0+\sum_{i=1}^N a_i(\alpha_i+1)s^{\alpha_i} 
\eeqs
and, by the fact $0<1-\alpha< 1+\alpha_i<1+\alpha_N$, have
\beqs
(1-\alpha)g(s)\le g(s)+ s g'(s)\le (1+\alpha_N)g(s). 
\eeqs
Then  
\beqs
 -\beta_2= \frac 1{1+\alpha_N} -1\le\frac{g(s)}{ g(s)+ s g'(s)}-1\le \frac{1}{1-\alpha}-1=\beta_1.
\eeqs
Therefore, inequality  \eqref{Kderest} follows this and \eqref{Kprime}. 

If $K=\bar K$ then 
\beq\label{Kbder}
\bar K'(\xi)= 
\begin{cases}
M_1\beta_1\xi^{\beta_1-1}=\beta_1 \bar K(\xi)/\xi,&\text{if } 0<\xi<Z_1,\\
0,&\text{if }  Z_1<\xi<Z_2,\\
K'_F(\xi),&\text{if }  \xi>Z_2.
\end{cases}
\eeq

We recall from \cite{ABHI1} that
\beq\label{Fder}
-\beta_2 \frac{K_F(\xi)}{\xi} \le K_F'(\xi)\le 0.
\eeq
Then the relation  \eqref{Kderest} obviously follows \eqref{Kbder} and \eqref{Fder}  for $0<\xi\ne Z_1,Z_2$.

If $K=\hat{K}$ then
\beq
\frac{\beta_1}{\xi} \ge \frac{K'}{K}=(\ln K)'
=\frac{\beta_1}\xi-\frac{b\beta_1\xi^{\beta_1-1}}{1+b\xi^{\beta_1}}-\frac{c\beta_2\xi^{\beta_2-1}}{1+c\xi^{\beta_2}}
\ge -\frac{\beta_2}{\xi}\frac{c\xi^{\beta_2}}{1+c\xi^{\beta_2}}\ge -\frac{\beta_2}{\xi}.
\eeq
This leads to \eqref{Kderest}.

Finally, consider $K=K_M$.
Write
$K_M(\xi)=K_F(\xi)M(\xi)$, where $M(\xi)=\bar{k}\xi^{\beta_1}/ (1+\bar{k}\xi^{\beta_1})$.

On the one hand,  $M'(\xi)\ge 0$, hence
\beqs
K_M'(\xi)=K_F'(\xi)M(\xi)+K_F(\xi)M'(\xi)\ge K_F'(\xi)M(\xi) \ge -\beta_2 \frac{K_F(\xi) M(\xi)}{\xi}=-\beta_2\frac{K_M(\xi)}{\xi}.
\eeqs

On the other hand, $K_F'(\xi)\le 0$ by \eqref{Fder}, and
\beqs
M'(\xi)=M(\xi)\Big(\frac{\beta_1}{\xi}-\frac{\bar k \beta_1 \xi^{\beta_1-1}}{1+\bar k \xi^{\beta_1}}\Big)
\le \frac{\beta_1}{\xi} M(\xi) ,
\eeqs
hence
\beqs
K_M'(\xi)=K_F'(\xi)M(\xi)+K_F(\xi)M'(\xi)  \le K_F(\xi) M'(\xi) \le  \frac{\beta_1}{\xi} K_F(\xi) M(\xi)=\beta_1 \frac{K_M(\xi)}{\xi}.
\eeqs
 
The proof is complete.
\end{proof}

\begin{corollary}\label{incor}
 Let $K=\bar K$, $K_I$, $\hat{K}$, $K_M$. Then the function $\xi^{m} K(\xi)$ is increasing on $[0,\infty)$ for any  real number $m\ge \beta_2$.
\end{corollary}
\begin{proof}
First,  consider $K=K_I$, $\hat{K}$, or $K_M$. According to \eqref{Kderest},
\beqs
(\xi^{m} K(\xi))'=m\xi^{m-1}K(\xi)+\xi^{m} K'(\xi)\ge m\xi^{m-1}K(\xi)-\beta_2\xi^{m-1}K(\xi)\ge 0
\eeqs
for any $\xi>0$.
Hence $\xi^m K(\xi)$ is increasing on $[0,\infty)$. If $K=\bar K$, then the statement is true on intervals not having $Z_1$ or $Z_2$ as an interior point. By continuity of $\xi^m \bar K(\xi)$ on $[0,\infty)$, the statement then holds true on $[0,\infty)$.
\end{proof}

\begin{lemma}[Monotonicity]\label{lemmono}
Let  $K=\bar K$, $K_I$, $\hat{K}$, $K_M$,  then 
\beq\label{Kmono}
\big (K(|y'|)y'-K(|y|)y\big )\cdot (y'-y) \ge \frac{d_5|y-y'|^{2+\beta_1}}{(1+|y|+|y'|)^{\beta_1+\beta_2}}
\quad \forall y,y'\in \mathbb R^n,
\eeq
where 
\beqs
d_5=\frac{ d_2(1-\beta_2)}{2^{\beta_1+1}(\beta_1+1)} .
\eeqs
\end{lemma}
\begin{proof}
Let $y\neq y'$ and denote by $[y,y']$ the line segment connecting $y$ and $y'$. 

\textbf{Case 1:} The origin does not belong to $[y,y']$. We  parametrize $[y,y']$ by
\beqs
\gamma(t)=ty'+(1-t)y\quad \text{for }t \in[0,1].
\eeqs

Define 
\beqs
h(t)=K(|\gamma(t)|)\gamma(t)\cdot (y'-y) \quad \text{for }    t\in[0,1].
\eeqs

 In case $K=K_I$, $\hat{K}$, $K_M$, function $h(t)\in C^1([0,1])$.
 When  $K=\bar K$, $h(t)$ is continuous on $[0,1]$ and $h'(t)$ is piecewise continuous on $[0,1]$ with at most four points of jump discontinuity at which $|\gamma(t)|=Z_1$ or $Z_2$. 
 By fundamental theorem of calculus,
\beq\label{Idef}
I\eqdef [K(|y'|)y'-K(|y|)y] \cdot (y'-y) = h(1)-h(0) =\int_0^1 h'(t) dt.
\eeq
At $t$ where $h'(t)$ exists, we calculate
\beqs
h'(t)=K(|\gamma(t)|)|y'-y|^2+K'(|\gamma(t)|)\frac{|\gamma(t)\cdot (y'-y)|^2}{|\gamma(t)|}.
\eeqs
By \eqref{Kderest} and Cauchy-Schwarz inequality
\begin{align*}
h'(t)
&\ge K(|\gamma(t)|)|y'-y|^2-\beta_2 \frac{K(|\gamma(t)|)}{|\gamma(t)|} \frac{|\gamma(t)\cdot (y'-y)|^2}{|\gamma(t)|}\\
&\ge K(|\gamma(t)|)|y'-y|^2-\beta_2 \frac{K(|\gamma(t)|)}{|\gamma(t)|} \frac{|\gamma(t)|^2|y'-y|^2}{|\gamma(t)|}
=(1-\beta_2)K(|\gamma(t)|)|y'-y|^2.
\end{align*}

Applying Lemma \ref{lem21}  and triangle inequality $|\gamma(t)|\le |y|+|y'|$, we infer
\beq\label{hp}
h'(t)
\ge  (1-\beta_2)|y'-y|^2  \frac{d_2 |\gamma(t)|^{\beta_1}}{(1+|\gamma(t)|)^{\beta_1+\beta_2} } 
\ge  \frac{d_2 (1-\beta_2)|y'-y|^2}{(1+|y|+|y'|)^{\beta_1+\beta_2}} |\gamma(t)|^{\beta_1}.
\eeq

Together with \eqref{Idef}, it implies 
\beq\label{Igam1}
I \ge  \frac{d_2 (1-\beta_2)|y'-y|^2}{(1+|y|+|y'|)^{\beta_1+\beta_2}} \int_0^1  |\gamma(t)|^{\beta_1} dt.
\eeq


It remains to estimate the last integral.
Let $z_0=\frac{y'+y}{2|y'-y|}$ and $u=\frac{y'-y}{|y'-y|}$.
Note that $|u|=1$. Then we write 
\begin{align*}
\int_0^1  |\gamma(t)|^{\beta_1} dt 
&=|y'-y|^{\beta_1} \int_0^1 |z_0+(t-\frac12)u|^{\beta_1}dt\\
&=|y'-y|^{\beta_1} \int_0^1 \Big(|z_0|^2+2(t-\frac12)z_0\cdot u +(t-\frac12)^2\Big)^{\beta_1/2}dt.
\end{align*}

If $z_0\cdot u\ge 0$ then
\beq\label{ig1}
\int_0^1  |\gamma(t)|^{\beta_1} dt \ge |y'-y|^{\beta_1} \int_{1/2}^1 \Big(t-\frac12\Big)^{\beta_1}dt
= \frac{|y'-y|^{\beta_1}}{2^{\beta_1+1}(\beta_1+1) }  .
\eeq

If $z_0\cdot u< 0$ then
\beq\label{ig2}
\int_0^1  |\gamma(t)|^{\beta_1} dt \ge |y'-y|^{\beta_1} \int_0^{1/2} \Big(\frac12-t\Big)^{\beta_1}dt
= \frac{|y'-y|^{\beta_1}}{2^{\beta_1+1}(\beta_1+1) } .
\eeq

In both cases, we have
\beq\label{Igam2}
I\ge \frac{ d_2(1-\beta_2)}{2^{\beta_1+1}(\beta_1+1)}\cdot \frac{|y'-y|^{\beta_1+2}}{(1+|y|+|y'|)^{\beta_1+\beta_2}},
\eeq
which proves \eqref{Kmono}.

\textbf{Case 2:} The origin belongs to $[y,y']$. We replace $y'$ by some $y_{\epsilon} \neq 0$ such that $0 \not\in[y,y_{\epsilon}]$, and $y_{\epsilon} \to y'$ as $\epsilon \to 0$.
Then let apply the inequality established in Case 1 for $y$ and $y_{\epsilon}$, then let $\epsilon \to 0$. 
\end{proof}

\begin{remark}
Our proof of \eqref{Igam2} from \eqref{Igam1} simplifies  DiBenedetto's arguments in \cite{DiDegenerateBook}, p. 13, 14. 
\end{remark}


\textbf{Degree Condition:} One of the following equivalent conditions
\beqs
 \deg(g)\le \frac{4}{n-2},\quad  \beta_2\le \frac 4{n+2},\quad  2\le (2-\beta_2)^*=\frac{(2-\beta_2)n}{n-2+\beta_2},\quad  2-\beta_2\ge \frac{2n}{n+2}.
\eeqs
(Above, $(2-\beta_2)^*$ is the Sobolev exponent corresponding to $2-\beta_2$.)

Hereafter, we assume the Degree Condition. Then the Sobolev space $W^{1,2-\beta_2}(U)$ is continuously embedded into $L^2(U)$. Also, the Poincar\'e-Sobolev inequality
\beq\label{PSi}
\|u\|_{L^2(U)}\le C_{\rm PS}\|\nabla u\|_{L^{2-\beta_2}(U)}
\eeq
holds for all functions $u\in W^{1,2-\beta_2}(U)$ which vanish on the boundary $\Gamma$, where $C_{\rm PS}$ is a positive constant.

In statements and calculations throughout, we use short-hand writing $\| \cdot\|=\|\cdot\|_{L^2(U)}$ and 
$\| u(t)\|_{L^p}=\|u(\cdot,t)\|_{L^p(U)}$ for a function $u(x,t)$ of $x$ and $t$.


\section{The IBVP and estimates of its solutions}\label{boundsec}

Let $K(\xi)$ be one of the functions $\bar{K}(\xi)$, $K_I(\xi)$, $\hat{K}(\xi)$, $K_M(\xi)$. 
Consider the following IBVP for the main PDE \eqref{peq}:
\beq\label{IBVP}
\begin{aligned}
\begin{cases}
p_t=\nabla\cdot (K(|\nabla p|)\nabla p)  &\text{in }  U\times (0,\infty),\\ 
p(x,0)=p_0(x), &\text{in }  U\\
p=\psi(x,t), &\text{on }  \partial U\times (0,\infty).
\end{cases}
\end{aligned}
\eeq

Dealing with the boundary condition, let $\Psi(x,t)$ be an extension of $\psi$ from $x\in\Gamma$ to $x\in \bar U$.

 Let $\bar p=p-\Psi$.  Then 
\beq\label{IBVP2}
\begin{cases}
\bar p_t=\nabla\cdot (K(|\nabla p|)\nabla p)-\Psi_t &\text{in }  U\times (0,\infty),\\ 
\bar p(x,0)=p_0(x)-\Psi(x,0), &\text{in }  U\\
\bar p=0, &\text {on } \partial U\times (0,\infty).
\end{cases}
\eeq

We will focus on estimates for $\bar p(x,t)$. The estimates for $p(x,t)$ can be obtained by simply using the triangle inequality
$$|p(x,t)|\le |\bar p(x,t)|+|\Psi(x,t)|.$$

Also, our results are stated in terms of $\Psi(x,t)$. These can be rewritten in terms of $\psi(x,t)$ by using a specific extension. For instance, the harmonic extension is utilized in \cite{HI1} with the use of norm relations in \cite{JerisonKenig1995}.

Throughout the paper, we will frequently use the following basic inequalities.
By Young's inequality, we have
\beq\label{bi3}
x^\beta\le x^{\gamma_1}+x^{\gamma_2}\quad\text{for all }x>0,\ \gamma_1\le \beta \le \gamma_2,
\eeq
\beq\label{bi2}
x^\beta\le 1+x^\gamma\quad\text{for all }x\ge 0, \ \gamma\ge\beta > 0.
\eeq

For any $r\ge1$, $x_1,x_2,\ldots,x_k\ge 0$, and $a,b\in\mathbb{R}^n$,
\beq\label{bi0}
(x_1+x_2+\ldots+x_k)^r\le k^{r-1}(x_1^r+x_2^r+\ldots+x_k^r), 
\eeq
\beq\label{bi1}
|a-b|^{r} \ge 2^{1-r}|a|^{r}-|b|^{r}.
\eeq

We also recall here a useful inequality from \cite{HIKS1}.

\begin{definition}\label{Env}
Given a function $f(t)$ defined on  $I=[0,\infty)$. 
We denote by $Env(f)$ a continuous, increasing  function $F(t)$ on $I$ such that $F(t) \ge f(t)$ for all $t \in I$.
\end{definition}


 \begin{lemma}[\cite{HIKS1}, Lemma 2.7] \label{ODE2}
Let $\theta>0$ and let $y(t)\ge 0, h(t)>0, f(t)\ge 0$ be continuous functions on $[0,\infty)$ that satisfy
 \beqs
 y'(t)\le -h(t)y(t)^\theta +f(t)\quad \text{for all } t>0.
 \eeqs
Then
 \beq\label{ubode}
 y(t)\le y(0)+\big[Env(f(t)/h(t))\big]^\frac{1}{\theta}\text{ for all } t\ge 0.
 \eeq
If $\int_0^\infty h(t)dt=\infty$ then
 \beq\label{ulode}
 \limsup_{t\rightarrow\infty} y(t)\le \limsup_{t\rightarrow\infty} \big[f(t)/h(t)\big]^\frac{1}{\theta}.
\eeq
 \end{lemma}


\noindent\textbf{Notation for constants.} In this section and section \ref{dependsec} below, the symbol $C$ denotes a \emph{generic} positive constant independent of the initial and boundary data; it may depend on the function $g(s)$ and the Poincar\'e-Sobolev constant $C_{PS}$ in \eqref{PSi}. In a particular proof, $C_0,C_1,\dots$ denote  positive constants of this type  but having their values fixed.

\subsection{Energy estimates }\label{L2sec}

In this subsection, we obtain $L^2$-estimates for the solution $p(x,t)$ for all time $t\ge 0$ and for $t\to\infty$.

\begin{theorem}\label{Lem31} 
There exists a positive constant $C$ such that for all $t\ge 0,$
\beq\label{west}
 \|\bar p(t)\|^2\le \|\bar p(0)\|^2+ C \big[1+Env f(t)\big]^\frac{2}{2-\beta_2},
 \eeq
where 
\beq \label{fdef}
f(t)=f[\Psi](t)\eqdef \norm{\nabla \Psi(t)}^2+\norm{\Psi_t (t)}^{\frac{2-\beta_2}{1-\beta_2}}.
\eeq
Furthermore, 
\beq\label{nodelta}
\begin{aligned}
\limsup_{t\to \infty} \|\bar p(t)\|^2&\le C (1+\limsup_{t\to\infty} f(t))^\frac{ 2}{2-\beta_2}. 
\end{aligned}
\eeq
\end{theorem}
\begin{proof}
Multiplying both sides of first equation in \eqref{IBVP2} by $\bar p$, integrating over the domain $U$ and using integration by parts we find that
\beq\label{weq}
\begin{aligned}
\frac 12\frac d{dt}\|\bar p\|^2 &=-\int_UK(|\nabla p|)\nabla p\cdot \nabla \bar p dx-\int_U\Psi_t \bar pdx\\
&= -\int_UK(|\nabla p|)|\nabla p|^2dx+\int_U K(|\nabla p|)\nabla p \cdot \nabla\Psi dx-\int_U\Psi_t \bar pdx.
\end{aligned}
\eeq

Using Cauchy's inequality and the bound \eqref{Kbdd} for function $K(\cdot)$, we obtain   
\begin{align*}
\int_U K(|\nabla p|)\nabla p \cdot \nabla\Psi dx 
&\le \frac12\int_UK(|\nabla p|)|\nabla p|^2dx + 
\frac12\int_UK(|\nabla p|)|\nabla \Psi|^2dx\\
&\le \frac12\int_UK(|\nabla p|)|\nabla p|^2dx +  C\norm{\nabla\Psi}^2.
\end{align*}

Let $\varep>0$. By H\"older's and Young's inequalities, 
\beqs
-\int_U\Psi_t \bar pdx \le\|\bar p\| \norm{\Psi_t} \le  \varep\|\bar p\|^{2-\beta_2}+C\varep^{-\frac{1}{1-\beta_2}}\norm{\Psi_t}^{\frac{2-\beta_2}{1-\beta_2}}.
\eeqs
Therefore,
\beqs
\frac d{dt}\|\bar p\|^2 \le -\int_UK(|\nabla p|)|\nabla p|^2dx+2\varep\|\bar p\|^{2-\beta_2}+C\norm{\nabla\Psi}^2+C\varep^{-\frac{1}{1-\beta_2}}\norm{\Psi_t}^{\frac{2-\beta_2}{1-\beta_2}}.
\eeqs

Let $\delta \in(0,1]$. By virtue of \eqref{mc2} and applying \eqref{bi1} to $r=2-\beta_2$, $a=\nabla \bar p$, $b=-\nabla \Psi$, we have 
\begin{align*}
K(|\nabla p|)|\nabla p|^2
&\ge d_2 \Big(\frac{\delta}{1+\delta}\Big)^{\beta_1+\beta_2} (|\nabla p|^{2-\beta_2}-\delta^{2-\beta_2}) \\
&\ge d_2\Big(\frac \delta {1+\delta} \Big)^{\beta_1+\beta_2}\Big(2^{\beta_2-1}|\nabla \bar p|^{2-\beta_2}-|\nabla \Psi|^{2-\beta_2}-\delta^{2-\beta_2}\Big)\\
&\ge \frac{d_2\delta^{\beta_1+\beta_2}}{2^{\beta_1+1}}|\nabla \bar p|^{2-\beta_2}- d_2|\nabla \Psi|^{2-\beta_2}-d_2\delta^{2+\beta_1}.
\end{align*}

Hence, we obtain  
\begin{multline}\label{psi-delta}
\frac d{dt}\norm {\bar p}^2
\le - \frac{d_2\delta^{\beta_1+\beta_2}}{2^{\beta_1+1}}\int_{U}|\nabla \bar p|^{2-\beta_2}dx+C\int_{U}|\nabla \Psi|^{2-\beta_2}dx + C{\delta^{2+\beta_1}}\\
+2\varep\|\bar p\|^{2-\beta_2} +C\norm{\nabla\Psi}^2 +C\varep^{-\frac{1}{1-\beta_2}}\norm{\Psi_t}^{\frac{2-\beta_2}{1-\beta_2}}.
\end{multline}

Using Poincar\'e--Sobolev's inequality \eqref{PSi}, we bound $\int_{U}|\nabla \bar p|^{2-\beta_2}dx$ from below by
\beq\label{pP}\int_{U}|\nabla \bar p|^{2-\beta_2}dx \ge \frac{   \|\bar p\|^{2-\beta_2}}{C_{\rm PS}^{2-\beta_2} }.
\eeq

For comparison of $\nabla \Psi$-terms on the right-hand side of \eqref{psi-delta}, applying H\"older's inequality gives  
 $$
\int_{U}|\nabla \Psi|^{2-\beta_2}dx \le C \|\nabla \Psi\|^{2-\beta_2}.
  $$ 
  
Then  we have from \eqref{psi-delta} that
\beqs
\frac d{dt}\norm {\bar p}^2 
\le - \Big(\frac {d_2 \delta^{\beta_1+\beta_2} } {2^{1+\beta_1} C_{\rm PS}^{2-\beta_2} }  - 2\varep\Big) \|\bar p\|^{2-\beta_2}+ C\Big({\delta^{2+\beta_1}}+\|\nabla \Psi\|^{2-\beta_2}+\norm{\nabla\Psi}^2+\varep^{-\frac{1}{1-\beta_2}}\norm{\Psi_t}^{\frac{2-\beta_2}{1-\beta_2}}\Big) .
\eeqs

Selecting $\varep =d_2 \delta^{\beta_1+\beta_2} /(2^{3+\beta_1} C_{\rm PS}^{2-\beta_2})$  yields 
\beq\label{psi-delta-2}
\frac d{dt}\norm {\bar p}^2 
\le - \frac {d_2 \delta^{\beta_1+\beta_2} } {2^{2+\beta_1} C_{\rm PS}^{2-\beta_2} }   \|\bar p\|^{2-\beta_2}+ C\Big({\delta^{2+\beta_1}}+\|\nabla \Psi\|^{2-\beta_2}+\norm{\nabla\Psi}^2+\delta^{-\frac{\beta_1+\beta_2}{1-\beta_2}}\norm{\Psi_t}^{\frac{2-\beta_2}{1-\beta_2}}\Big) .
\eeq

Denote $y(t)=\|\bar p(t)\|^2$. We rewrite \eqref{psi-delta-2} as
\beq\label{ydif}
\frac {dy}{dt} 
\le - C_0 \delta^{\beta_1+\beta_2} y^\frac{2-\beta_2}2+ C\Big({\delta^{2+\beta_1}}+\|\nabla \Psi\|^{2-\beta_2}+\norm{\nabla\Psi}^2 +\delta^{-\frac{\beta_1+\beta_2}{1-\beta_2}}\norm{\Psi_t}^{\frac{2-\beta_2}{1-\beta_2}}\Big),
\eeq
where
\beqs
C_0=\frac {d_2} {2^{2+\beta_1} C_{\rm PS}^{2-\beta_2} } .
\eeqs

Select  $\delta =1$. On the RHS of \eqref{ydif}, we apply inequality \eqref{bi2} to have
\beq\label{psH}
 \|\nabla \Psi\|^{2-\beta_2}\le 1 +\|\nabla \Psi\|^2 .
\eeq 
 Then
\beq\label{pd3}
\frac {d y}{dt}
\le - C_0  y^\frac{2-\beta_2}2+ C(1+f(t)).
\eeq
Applying \eqref{ubode} and  \eqref{ulode}  in Lemma \ref{ODE2} to \eqref{pd3}, we obtain    \eqref{west} and \eqref{nodelta}, respectively.
\end{proof}

In case the boundary data is asymptotically small as $t\to\infty$, we prove in the next theorem that so is $\|\bar p(t)\|$.

\begin{theorem}\label{pasmall}
For any $\varep>0$, there is $\delta_0>0$ such that if 
\beq\label{small0}
\limsup_{t\to\infty} \|\nabla \Psi(t)\|\le \delta_0 \quad \text{and} \quad  \limsup_{t\to\infty} \| \Psi_t( t)\|\le \delta_0,
\eeq
then 
\beq\label{wlim0}
 \limsup_{t\to\infty}\| \bar p(t) \| \le  \varep .
\eeq

Consequently, if 
\beq\label{limcond}
\lim_{t\to\infty} \|\nabla \Psi(t)\|= \lim_{t\to\infty} \| \Psi_t( t)\|=0,
\eeq
 then
\beq\label{wlim}
\lim_{t\to\infty}\| \bar p(t) \| = 0.
\eeq
\end{theorem}
\begin{proof}
Let $\delta\in(0,1]$. 
Applying \eqref{ulode} in Lemma \ref{ODE2} to \eqref{ydif}, and then using inequality \eqref{bi0} with $r=2/(2-\beta_2)$ give      
\begin{align}\label{f-delta}
&\limsup_{t\to \infty} \|\bar p(t)\|^2 \notag \\
&\le C   \Big\{\delta^{2-\beta_2} + \limsup_{t\to\infty} \Big[\delta^{-(\beta_1+\beta_2)}\big(\|\nabla \Psi\|^{2-\beta_2}+\norm{\nabla\Psi}^2\big)+\delta^{-\frac{(\beta_1+\beta_2)(2-\beta_2)}{1-\beta_2} } \norm{\Psi_t}^{\frac{2-\beta_2}{1-\beta_2}}\Big]    \Big\}^\frac{ 2}{2-\beta_2} \notag \\
&\le C   \Big\{\delta^{2} + \limsup_{t\to\infty} \Big[\delta^\frac{-2(\beta_1+\beta_2)}{2-\beta_2}\big(\|\nabla \Psi\|^2+\norm{\nabla\Psi}^\frac{4}{2-\beta_2}\big)+\delta^{-\frac{2(\beta_1+\beta_2)}{1-\beta_2} } \norm{\Psi_t}^{\frac{2}{1-\beta_2}}\Big]    \Big\}. 
\end{align}

Assume \eqref{small0} with $0<\delta_0\le 1$. It follows \eqref{f-delta} that 
 \beq\label{small1}
\limsup_{t\to \infty} \|\bar p(t)\|^2
\le C_1 \Big(\delta^{2}+3\delta^{-\frac{2(\beta_1+\beta_2)}{1-\beta_2} } \delta_0^2   \Big)
\eeq
for some $C_1>0$.
We choose $\delta$ sufficiently small so that $C_1 \delta^2\le \varep^2/2$, and then, with such $\delta$,  choose $\delta_0$ to satisfy 
\beqs
3C_1  \delta^{-\frac{2(\beta_1+\beta_2)}{1-\beta_2} } \delta_0^2\le 
  \varep^2/2.
\eeqs   
Therefore, the desired estimate \eqref{wlim0} follows  \eqref{small1}. 

In the case \eqref{limcond} is satisfied, we have  \eqref{wlim0} holds for any $\varep>0$, which implies \eqref{wlim}.
\end{proof}

\subsection{Gradient estimates}\label{gradsec}

This subsection is focused on estimating the $L^{2-\beta_2}$-norm for $\nabla p(x,t)$.
The following function $H(\xi)$ will be crucial in our gradient estimates.

\begin{definition}
Define for $\xi\ge 0$, the auxiliary function
\beq\label{Hdef}
H(\xi)=\int_0^{\xi^2} K(\sqrt s) ds.
\eeq
\end{definition}

We compare $H(\xi)$ with $K(\xi)\xi^2$ and $\xi^{2-\beta_2}$ in the following lemma.

\begin{lemma}
For any $\xi\ge 0$,
\beq\label{HtoK}
  \frac{d_2}{d_3} K(\xi)\xi^2 -d_2 K_*(\xi_c) \xi_c^2\le H(\xi)\le 2 K(\xi)\xi^2.
\eeq

For any $\delta>0$ and $\xi\ge 0$,
\beq\label{HK2}
d_2 \Big(\frac\delta{1+\delta}\Big)^{\beta_1+\beta_2} \big(\xi^{2-\beta_2}-\delta^{2-\beta_2}\big)\le  H(\xi)
  \le   2d_3 \xi^{2-\beta_2}.
\eeq

\end{lemma}

\begin{proof}
By Corollary \ref{incor}, the function $K(\xi)\xi$ is increasing, hence we have
\beq\label{HKtwice}
H(\xi)=2\int_0^{\xi} K(s)s ds\le 2K(\xi)\xi\int_0^\xi 1ds=2K(\xi)\xi^2.
\eeq
This proves the second inequality of \eqref{HtoK}.
Combining this with the second inequality in \eqref{mc2} for $m=2$ yields the second inequality in \eqref{HK2}.

By Lemma \ref{lem21},
\beq
H(\xi)\ge d_2 \int_0^{\xi^2} K_*(\sqrt s) ds.
\eeq


For $\xi>\xi_c$, using properties (P2) and (P3) of $K_*(\xi)$ in section \ref{presec}, we have
\beq
H(\xi)\ge d_2K_*(\xi) \int_{\xi_c^2}^{\xi^2} 1 ds
= d_2 K_*(\xi)(\xi^2-\xi_c^2)
\ge d_2( K_*(\xi)\xi^2-K_*(\xi_c)\xi_c^2 ).
\eeq

For $\xi\le\xi_c$, according to Corollary~\ref{incor}  the function $\xi^2 K(\xi)$ is increasing thus  
\beq
 d_2( K_*(\xi)\xi^2-K_*(\xi_c)\xi_c^2 ) \le 0\le H(\xi).
\eeq
Combining the above two inequalities we have
\beq\label{HKstar}
 d_2(K_*(\xi)\xi^2 - K_*(\xi_c) \xi_c^2)\le H(\xi).
\eeq

Applying \eqref{mc1} in Lemma \ref{lem21} to compare $K_*(\xi)$ with $K(\xi)$ in \eqref{HKstar} yields
\beq
 \frac{d_2}{d_3} K(\xi)\xi^2 -d_2K_*(\xi_c)\xi_c^2  \le H(\xi).
\eeq
Hence, we obtain the first inequality of \eqref{HtoK}.

Next, we prove the first inequality of \eqref{HK2}.
Since it trivially holds true for all $\xi\le\delta$, it suffices to consider  $\xi>\delta$. 
From \eqref{HKtwice},   
\beqs
H(\xi)\ge 2 \int_{\delta}^{\xi}K(s)s ds = 2 \int_{\delta}^{\xi}K(s)s^{\beta_2}s^{1-\beta_2} ds.
\eeqs  

According to  Corollary~\ref{incor}, the function $K(s)s^{\beta_2}$ is increasing thus 
\beqs
H(\xi)\ge 2 K(\delta)\delta^{\beta_2} \int_{\delta}^{\xi}s^{1-\beta_2}ds =\frac{2}{2-\beta_2} K(\delta)\delta^{\beta_2}\big(\xi^{2-\beta_2}-\delta^{2-\beta_2}\big)
\ge K(\delta)\delta^{\beta_2}\big(\xi^{2-\beta_2}-\delta^{2-\beta_2}\big),
\eeqs
which, together with \eqref{mc1}, proves the first inequality of \eqref{HK2}.
 The proof is complete.
\end{proof}

Bounds for the gradient in terms of the initial and boundary data are obtained in the next theorem.

\begin{theorem}\label{grad-est} 
For all $t\ge 0$, 
\beq\label{Hest2}
\begin{aligned}
\int_U |\nabla p(x,t)|^{2-\beta_2}dx 
&\le  C\Big(1+\|\bar p(0)\|^2 +e^{-\frac t 2} \int_U |\nabla p(x,0)|^{2-\beta_2}dx \\
&\quad + \big[Env f(t) \big]^\frac{2}{2-\beta_2}+\int_0^t e^{-\frac 12(t-\tau)} \norm{\nabla\Psi_t(\tau)}^2d\tau\Big).
\end{aligned}
\eeq

Furthermore, 
\beq\label{limsupG}
\limsup_{t\to\infty} \int_U |\nabla  p(x,t)|^{2-\beta_2}dx \le C ( 1+ \limsup_{t\to\infty}G_1(t)),
\eeq
where
\beqs
G_1(t) = G_1[\Psi](t) \eqdef f(t)^\frac{2}{2-\beta_2}+\norm{\nabla\Psi_t(t)}^2. 
\eeqs
\end{theorem}
\begin{proof}
Multiplying the first equation in the \eqref{IBVP2} by $\bar p_t$, integrating over the domain $U$, using integration by parts for the first integral on the RHS and by the fact that $\bar{p}_t=p_t-\Psi_t$;  we have
\begin{align*}
\int_U \bar p_t^2 dx 
&= -\int_UK(|\nabla p|) \nabla p \cdot \nabla \bar p_tdx-\int_U \bar p_t\Psi_t\\
&= -\int_UK(|\nabla p|) \nabla p \cdot \nabla p_tdx +\int_UK(|\nabla p|) \nabla p \cdot \nabla\Psi_tdx-\int_U \bar p_t\Psi_t dx.
\end{align*}

For the first integral on the RHS, using definition \eqref{Hdef} of $H(\xi)$ we have
\beq\label{wteq}
\|\bar p_t\|^2  +\frac 12\frac d{dt}\int_U {\mathcal H}(x,t)dx =\int_UK(|\nabla p|) \nabla p \cdot \nabla\Psi_tdx-\int_U \bar p_t\Psi_t dx,
\eeq
where, for the sake of simplicity, we denoted
$${\mathcal H}(x,t)=H(|\nabla p(x,t)|).$$

 Let 
 \beq\label{Edef}
\mathcal E(t) = \int_U |\bar p(x,t)|^2 dx +\int_U {\mathcal H}(x,t) dx.
\eeq

Summing \eqref{weq} and \eqref{wteq} gives    
\beq\label{gradeq}
\begin{aligned} 
\|\bar p_t\|^2 +\frac 12\frac d{dt}\mathcal E(t) &= -\int_UK(|\nabla p|)|\nabla p|^2dx +\int_U K(|\nabla p|)\nabla p \cdot \nabla(\Psi +\Psi_t) dx\\
&\quad  -\int_U\Psi_t (\bar p +\bar p_t) dx=I_1+I_2+I_3. 
\end{aligned}
\eeq

It follows from \eqref{HtoK} that 
\beq\label{I1est}
I_1\le - \frac 1 2\int_U {\mathcal H}(x,t) dx.
\eeq

Let $\varep>0$. 
Applying Cauchy's inequality for $I_2$, and using the fact \eqref{Kbdd} that $K(\cdot)$ bounded   
\begin{align*}
|I_2|\le \varep \int_UK(|\nabla p|)|\nabla p|^2dx+ C\varep^{-1} (\norm{\nabla\Psi}^2+\norm{\nabla\Psi_t}^2).
\end{align*}


Again using \eqref{HtoK},       
\beq\label{I2est}
|I_2|\le \varep \int_U(C_1 {\mathcal H}(x,t)+ C_2)dx+ C\varep^{-1} (\norm{\nabla\Psi}^2+ \norm{\nabla\Psi_t}^2),
\eeq
where $C_1=d_3/d_2$ and $C_2=d_3K_*(\xi_c)\xi_c^2$. 

For $I_3$, applying Cauchy's inequality gives
\beq\label{I3est}
|I_3|\le \frac 12 (\|\bar p\|^2 + \|\bar p_t\|^2 )+\norm{ \Psi_t}^2.
\eeq

Combining \eqref{gradeq}--\eqref{I3est}, we obtain  
\beq\label{Hw}
\frac d{dt}\mathcal E(t) +  \|\bar p_t\|^2 \le -(1-2\varep C_1) \int_U{\mathcal H}(x,t)dx+\|\bar p\|^2 +2\varep C_2|U| + C\varep^{-1}(\norm{\nabla\Psi}^2+\norm{\nabla\Psi_t}^2)+ 2\norm{\Psi_t}^2. 
\eeq
Selecting $\varep=\delta/(4C_1)$ in \eqref{Hw} with $\delta\in(0,1]$  and using \eqref{Edef}, we find that 
\begin{align}
 \frac d{dt}\mathcal E(t) +  \|\bar p_t\|^2 
&\le -\frac 12 \int _U {\mathcal H}(x,t)dx  + \|\bar p\|^2+C\delta + C\delta^{-1}(\norm{\nabla\Psi}^2+ \norm{\nabla \Psi_t}^2) + 2\norm{\Psi_t}^2 \nonumber \\
&\le -\frac 12 \mathcal E(t) + \frac 32 \|\bar p\|^2 +C\delta + C\delta^{-1}(\norm{\nabla\Psi}^2+ \norm{\nabla \Psi_t}^2) +2 \norm{\Psi_t}^2.\label{Hw2}
\end{align}

Letting $\delta=1$, 
\beq\label{Hsim}
 \frac d{dt}\mathcal E(t)  +  \|\bar p_t\|^2\\
\le -\frac 12 \mathcal E(t) + \frac 32 \|\bar p\|^2+C(1+\norm{\nabla\Psi}^2+ \norm{\nabla \Psi_t}^2+ \norm{\Psi_t}^2).
\eeq

Thanks to estimate \eqref{west} of $ \|\bar p(t)\|^2$ and by using \eqref{bi2} to bound 
\beq\label{pst1}
 \norm{\Psi_t}^2\le 1+ \norm{\Psi_t}^\frac{2-\beta_2}{1-\beta_2},\eeq
 we obtain  
 \beqs
\frac d{dt}\mathcal E(t) \le -\frac 12 \mathcal E(t)+ \frac 32 \|\bar p(0)\|^2+C+C(\big[Env f(t) \big]^\frac{2}{2-\beta_2}+\norm{\nabla\Psi_t(t)}^2).
 \eeqs

 It follows from Gronwall's inequality that
  \beqs
\mathcal E(t) \le e^{-\frac t 2}\mathcal E(0)+ 3 \|\bar p(0)\|^2 +C+C\int_0^t e^{-\frac 12(t-\tau)} (\big[Env f(\tau) \big]^\frac{2}{2-\beta_2}+\norm{\nabla\Psi_t(\tau)}^2)d\tau,
 \eeqs 
which, by monotonicity of  function $Env f(t)$ and definition of $\mathcal{E}(t)$, leads to
\begin{multline*}
 \int_U {\mathcal H}(x,t)dx \le  4\|\bar p(0)\|^2 +e^{-\frac t 2}\int_U {\mathcal H}(x,0) dx \\+C+C[Env f(t) \big]^\frac{2}{2-\beta_2}+C\int_0^t e^{-\frac 12(t-\tau)} \norm{\nabla\Psi_t(\tau)}^2d\tau.
\end{multline*}

We bound ${\mathcal H}(x,t)$ from below by the first inequality in \eqref{HK2} with $\delta=1$, and bound $\mathcal H(x,0)$ from above by  the second inequality of \eqref{HK2}. It results in
\begin{multline*}
C_3 \int_U |\nabla p(x,t)|^{2-\beta_2}dx - C_4\le  4\|\bar p(0)\|^2 +2d_3e^{-\frac t 2}\int_U |\nabla p(x,0)|^{2-\beta_2} dx \\
 +C+C[Env f(t) \big]^\frac{2}{2-\beta_2}+C\int_0^t e^{-\frac 12(t-\tau)} \norm{\nabla\Psi_t(\tau)}^2d\tau
\end{multline*}
for constants $C_3=d_2/2^{\beta_1+\beta_2}$ and $C_4=C_3|U|$, and  estimate \eqref{Hest2} follows.

Neglecting $ \|\bar p_t\|^2$ on the LHS from \eqref{Hsim}, we have 
\beq\label{Hsim4}
 \frac d{dt}\mathcal E(t)\\
\le -\frac 12 \mathcal E(t) + \frac 32 \|\bar p\|^2+C(1+\norm{\nabla\Psi}^2+ \norm{\nabla \Psi_t}^2+ \norm{\Psi_t}^2).
\eeq
Applying \eqref{ulode} in Lemma \ref{ODE2} to  \eqref{Hsim4}, we have
\beqs
\limsup_{t\to\infty} \mathcal E(t)\le  3\limsup_{t\to\infty}\|\bar p\|^2 +C\limsup_{t\to\infty}(1+\norm{\nabla\Psi}^2+\norm{\nabla\Psi_t}^2+\norm{\Psi_t}^2).
\eeqs
Combining this with \eqref{nodelta} yields 
\beq\label{limH0}
\limsup_{t\to\infty} \int_U {\mathcal H}(x,t)dx \le C ( 1+ \limsup_{t\to\infty}G_1(t)).
\eeq

Again, by using the first inequality in \eqref{HK2} with $\delta=1$ to bound ${\mathcal H}(x,t)$ from below in terms of $|\nabla p(x,t)|^{2-\beta_2}$, we obtain  estimate \eqref{limsupG} from \eqref{limH0}.
 \end{proof}


%

Below is a counterpart of Theorem \ref{pasmall}, but for the gradient instead.

\begin{theorem}\label{gradasmall}
For any $\varep>0$, there is $\delta_0>0$ such that if 
\beq\label{small2}
\limsup_{t\to\infty} (\|\nabla \Psi(t)\|+ \| \Psi_t( t)\|+ \|\nabla \Psi_t( t)\|)\le \delta_0 
\eeq
 then 
\beq\label{gradplim}
  \limsup_{t\to\infty}\int_U |\nabla p(x,t)|^{2-\beta_2} dx \le   \varep .
\eeq

Consequently, if 
\beq\label{small3}
\lim_{t\to\infty} \|\nabla \Psi(t)\|=\lim_{t\to\infty} \| \Psi_t( t)\|=\lim_{t\to\infty} \|\nabla \Psi_t(t)\|=0 
\eeq
then 
\beq\label{wlim2}
  \lim_{t\to\infty}\int_U |\nabla p(x,t)|^{2-\beta_2} dx =0.
\eeq
\end{theorem}
\begin{proof}
First, we estimate the limit superior of $ \int_U {\mathcal H}(x,t)dx$ as $t\to\infty$.
Let $\delta\in (0,1]$. Applying \eqref{ulode} in Lemma \ref{ODE2} to  \eqref{Hw2}, we have
\beq\label{Hw23}
\limsup_{t\to\infty} \mathcal E(t)\le  C\limsup_{t\to\infty}\|\bar p\|^2 +C\Big\{\delta^{-1}\limsup_{t\to\infty}(\norm{\nabla\Psi}^2+\norm{\nabla\Psi_t}^2)+  \limsup_{t\to\infty} \norm{\Psi_t}^2+ \delta\Big\}.
\eeq

It follows from \eqref{Hw23} and \eqref{f-delta} that
\beqs
\begin{split}
&\limsup_{t\to\infty} \int_U {\mathcal H}(x,t)dx \\
& \le C  \Big\{ \delta^2 + \delta^{-\frac{2(\beta_1+\beta_2)}{2-\beta_2} }\big(\limsup_{t\to\infty}\|\nabla \Psi\|^2+\limsup_{t\to\infty}\norm{\nabla\Psi}^{\frac{4}{2-\beta_2}}\big)+  \delta^{-\frac{2(\beta_1+\beta_2)}{1-\beta_2} } \limsup_{t\to\infty}\norm{\Psi_t}^{\frac{2}{1-\beta_2}}    \Big\}\\
&\quad +C\Big\{ \delta^{-1}\limsup_{t\to\infty}(\norm{\nabla\Psi}^2+\norm{\nabla\Psi_t}^2)+  \limsup_{t\to\infty} \norm{\Psi_t}^2+\delta\Big\}.
\end{split}
\eeqs
Thanks to the fact $\delta\le 1$, it follows that 
\beq\label{Hlim}
\begin{split}
\limsup_{t\to\infty} \int_U {\mathcal H}(x,t)dx 
&\le C  \Big\{ \delta   
 + \delta^{-\kappa } \limsup_{t\to\infty}\Big(\|\nabla \Psi(t)\|^2+\norm{\nabla\Psi(t)}^\frac{4}{2-\beta_2}\\
 &\qquad + \norm{\Psi_t(t)}^{\frac{2}{1-\beta_2}}  +   \norm{\Psi_t(t)}^2
 +\norm{\nabla\Psi_t(t)}^2\Big)\Big\},
\end{split}
\eeq 
where $\kappa=\max\Big\{\frac{2(\beta_1+\beta_2)}{1-\beta_2},1\Big\}$.

Let $\delta_1$ be any number in $(0,1]$.
Applying the first inequality of \eqref{HK2} with $\delta=\delta_1$ gives
\beq\label{Hd1} 
{\mathcal H}(x,t)\ge \frac{d_2 \delta_1^{\beta_1+\beta_2}}{2^{\beta_1+\beta_2}}|\nabla  p(x,t)|^{2-\beta_2}-d_2 \delta_1^{2+\beta_1}.
\eeq

Combining this with \eqref{Hlim} yields   
\begin{align}\label{limsupGrad}
\limsup_{t\to\infty} \int_U |\nabla  p(x,t)|^{2-\beta_2}dx 
&\le C\delta_1^{-(\beta_1+\beta_2)} \limsup_{t\to\infty} \int_U {\mathcal H}(x,t)dx +C \delta_1^{2-\beta_2} \notag \\
&\le C \delta_1^{-(\beta_1+\beta_2)} \Big\{ \delta+ \delta^{-\kappa} \limsup_{t\to\infty}\Big( \|\nabla \Psi\|^2+\norm{\nabla\Psi}^\frac{4}{2-\beta_2} \notag \\
&\quad +\norm{\Psi_t}^{\frac{2}{1-\beta_2}}   +  \norm{\Psi_t}^2  
 +\norm{\nabla\Psi_t}^2\Big)\Big\} +C \delta_1^{2-\beta_2}.
\end{align}

Assume \eqref{small2} with $\delta_0\in(0,1]$, then \eqref{limsupGrad} yields
\begin{align}
\limsup_{t\to\infty} \int_U |\nabla  p|^{2-\beta_2}dx
& \le C_5 \delta_1^{-(\beta_1+\beta_2)}  \Big\{ \delta+ \delta^{-\kappa} (\delta_0^2+\delta_0^\frac{4}{2-\beta_2}
+ \delta_0^{\frac{2}{1-\beta_2}} )\Big\}   + C_5\delta_1^{2-\beta_2} \nonumber\\
& \le C_5 \delta_1^{-(\beta_1+\beta_2)}  (\delta+ 3\delta^{-\kappa} \delta_0^2)  + C_5\delta_1^{2-\beta_2} \label{gpl}
\end{align}
for some $C_5>0$ is independent of $\delta$, $\delta_0$, and $\delta_1$. 

First we choose $\delta_1$ sufficiently small satisfying $ C_5\delta_1^{2-\beta_2}\le \varep/3$. With this $\delta_1$,  choose $\delta\in (0,1]$ such that 
$
C_5\delta_1^{-(\beta_1+\beta_2)} \delta\le \varep/ 3.
$
Next, we choose $\delta_0>0$ much smaller than $\delta_1, \delta$ that satisfies 
$$
C_5 \delta_1^{-(\beta_1+\beta_2)}   \delta^{-\kappa}  \delta_0^2  \le \varep/9.
$$
Then \eqref{gradplim} follows \eqref{gpl}. 
Finally, under condition \eqref{small3}, the estimate \eqref{gradplim} holds for all $\varep>0$, which proves \eqref{wlim2}.
\end{proof}

When time $t$ is large, we improve the estimates in Theorem \ref{grad-est} by deriving uniform Gronwall-type inequalities.


\begin{theorem}\label{theo39}  If $t\ge 1$ then
\beq\label{wteq8}
\int_{t-\frac 12}^t\|\bar p_t(\tau)\|^2 d\tau+ \int_U |\nabla p(x,t)|^{2-\beta_2}dx 
\le C \Big(1+\| \bar p(t-1)\|^2
+\int_{t-1}^{t} (f(\tau)+\norm{\nabla\Psi_t(\tau)}^2)d\tau\Big),
\eeq
and, consequently,
\beq\label{wteq9}
\int_U |\nabla p(x,t)|^{2-\beta_2} dx \le C\Big(1+\|\bar p(0)\|^2+ (Env f(t))^\frac{2}{2-\beta_2}+\int_{t-1}^{t}\norm{\nabla\Psi_t(\tau)}^2 d\tau\Big).
\eeq
\end{theorem}
\begin{proof}
On the right-hand side of \eqref{psi-delta}, we use \eqref{pP} again but this time to bound $\|\bar p\|$ in terms of $\int_U |\nabla \bar p|^{2-\beta_2}dx$. Then, with the same choice of $\varepsilon$,  we have instead of \eqref{psi-delta-2} \beq\label{psi-delta-3}
\frac d{dt}\norm {\bar p}^2 
\le - \frac {d_2 \delta^{\beta_1+\beta_2} } {2^{2+\beta_1}  }   \int_U |\nabla \bar p|^{2-\beta_2}dx + C\Big({\delta^{2+\beta_1}}+\|\nabla \Psi\|^{2-\beta_2}+\norm{\nabla\Psi}^2+\delta^{-\frac{\beta_1+\beta_2}{1-\beta_2}}\norm{\Psi_t}^{\frac{2-\beta_2}{1-\beta_2}}\Big).
\eeq

Integrating \eqref{psi-delta-3} in time from $t-1$ to $t$,  we have 
\begin{multline}\label{new1}
\norm {\bar p(t)}^2+C_1 \delta^{\beta_1+\beta_2}\int_{t-1}^t \int_{U}|\nabla \bar p|^{2-\beta_2}dxd\tau
\le \norm {\bar p(t-1)}^2 + C{\delta^{2+\beta_1}} \\
+C\int_{t-1}^t(\|\nabla \Psi\|^{2-\beta_2}+ \norm{\nabla\Psi}^2+\delta^{-\frac{\beta_1+\beta_2}{1-\beta_2}}\norm{\Psi_t}^{\frac{2-\beta_2}{1-\beta_2}})d\tau,
\end{multline}
where $C_1>0$ is independent of $\delta$.

Let $\varep\in(0,1]$.
Applying Cauchy's inequality to integrals on the RHS of \eqref{wteq}, we have the following: 
\beq\label{wteq1}
\|\bar p_t\|^2  +\frac 12\frac d{dt}\int_U {\mathcal H}(x,t)dx \le\varepsilon\int_UK(|\nabla p|) |\nabla p|^2dx+\frac{C}{\varepsilon}\int_UK(|\nabla p|) |\nabla\Psi_t|^2dx+\frac 12\|\bar p_t\|^2+\frac 12\norm{\Psi_t}^2 .
\eeq
Then using \eqref{Kbdd} for the second integral on the RHS of \eqref{wteq1}, we have
\beq\label{new2a}
\begin{aligned}
\|\bar p_t\|^2  + \frac d{dt}\int_U {\mathcal H}(x,t)dx &\le 2\varepsilon\int_UK(|\nabla p|) |\nabla p|^2dx+C(\varep^{-1}\norm{\nabla\Psi_t}^2+ \norm{\Psi_t}^2).
\end{aligned}
\eeq
By virtue of  \eqref{HtoK}, we find from \eqref{new2a} that
\beq\label{new3}
\begin{aligned}
\|\bar p_t\|^2  + \frac d{dt}\int_U {\mathcal H}(x,t)dx &\le C\varep \int_U{\mathcal H}(x,t)dx+C(\varep +\varep^{-1}\norm{\nabla\Psi_t}^2+ \norm{\Psi_t}^2).
\end{aligned}
\eeq

Integrating \eqref{new3} in time  from $s$ to $t$ where $s\in[t-1,t]$, we have 
\begin{align}
&\int_s^t\|\bar p_t\|^2 d\tau + \int_U {\mathcal H}(x,t)dx \notag \\
&\quad \le \int_U {\mathcal H}(x,s)dx+ C\varep\int_s^t\int_U{\mathcal H}(x,\tau)dxd\tau+ C\int_s^t(\varep +\varep^{-1}\norm{\nabla\Psi_t}^2+ \norm{\Psi_t}^2)d\tau \notag\\
&\quad \le \int_U {\mathcal H}(x,s)dx+C\int_{t-1}^t\int_U{\mathcal H}(x,t)dxd\tau + C\varep+C\int_{t-1}^t(\varep^{-1}\norm{\nabla\Psi_t}^2+ \norm{\Psi_t}^2)d\tau. \label{wteq5}
\end{align}

Integrating \eqref{wteq5} in $s$ from $t-1$ to $t$  shows that  
\beq\label{wteq6}
\int_{t-1}^t\int_s^t\|\bar p_t\|^2 d\tau ds + \int_U {\mathcal H}(x,t)dx 
\le C\Big\{\int_{t-1}^t\int_U{\mathcal H}(x,\tau)dxd\tau+\varep+\int_{t-1}^t(\varep^{-1}\norm{\nabla\Psi_t}^2+ \norm{\Psi_t}^2)d\tau\Big\}.
\eeq

Estimating first term of \eqref{wteq6} by 
\beqs
\int_{t-1}^t\int_s^t\|\bar p_t\|^2 d\tau ds  \ge \int_{t-1}^{t-\frac 12}\int_{t-\frac 12}^t\|\bar p_t\|^2  d\tau ds\ge  \frac 12\int_{t-\frac 12}^t\|\bar p_t\|^2 d\tau,
\eeqs
we get
\beq\label{wteq06}
\begin{aligned}
 \frac 12\int_{t-\frac 12}^t\|\bar p_t\|^2 d\tau + \int_U {\mathcal H}(x,t)dx 
&\le C\Big\{ \int_{t-1}^t\int_U{\mathcal H}(x,\tau)dxd\tau+\varep+\int_{t-1}^t(\varep^{-1}\norm{\nabla\Psi_t}^2+ \norm{\Psi_t}^2)d\tau\Big\}.
\end{aligned}
\eeq

Estimating the double integral on the RHS of  \eqref{wteq06} by combining the second inequality of \eqref{HK2} with \eqref{new1}, we obtain
\begin{multline}\label{new8}
\frac12 \int_{t-\frac 12}^t\|\bar p_t(\tau)\|^2 d\tau+ \int_U {\mathcal H}(x,t)dx\\
 \le C\Big\{\delta^{-(\beta_1+\beta_2)}\norm {\bar p(t-1)}^2 
+ \delta^{2-\beta_2}+\delta^{-(\beta_1+\beta_2)}\int_{t-1}^t (\|\nabla \Psi\|^{2-\beta_2}+ \norm{\nabla\Psi}^2+\delta^{-\frac{\beta_1+\beta_2}{1-\beta_2}}\norm{\Psi_t}^{\frac{2-\beta_2}{1-\beta_2}})d\tau\\
+\varep+\int_{t-1}^t(\varep^{-1}\norm{\nabla\Psi_t}^2+ \norm{\Psi_t}^2) d\tau \Big\}.
\end{multline}

In \eqref{new8}, choosing $\varep=\delta=1$,  using \eqref{psH} and \eqref{pst1} give 
\begin{align*}
\int_{t-1/2}^t \|\bar p_t\|^2 d\tau + \int_U {\mathcal H}(x,t)dx 
&\le C\Big(1+ \|\bar p(t-1)\|^2+\int_{t-1}^{t}(\norm{\nabla\Psi}^2+\norm{\nabla\Psi_t}^2+ \norm{\Psi_t}^\frac{2-\beta_2}{1-\beta_2} ) d\tau\Big) \notag \\
&= C \Big(1+\| \bar p(t-1)\|^2
+\int_{t-1}^{t} (f(\tau)+\norm{\nabla\Psi_t(\tau)}^2)d\tau\Big).
\end{align*}

Combining this with the first inequality in \eqref{HK2} with $\delta=1$ yields \eqref{wteq8}.

Combining \eqref{wteq8} with \eqref{west}, we obtain \eqref{wteq9}.
\end{proof}

As for $t\to\infty$, we have the following alternative results.

\begin{corollary}\label{cor310}
Ones have
\beq\label{wteq11}
\limsup_{t\to\infty} \int_U |\nabla p(x,t)|^{2-\beta_2}  dx \le C\Big(1+ \limsup_{t\to\infty} f(t)^\frac{2}{2-\beta_2}+\limsup_{t\to\infty}\int_{t-1}^{t}\norm{\nabla\Psi_t(\tau)}^2d\tau\Big).
\eeq
Moreover, if 
\beq\label{small30}
\lim_{t\to\infty} \|\nabla \Psi(t)\|=\lim_{t\to\infty} \| \Psi_t( t)\|= \lim_{t\to\infty} \int_{t-1}^t \|\nabla \Psi_t(\tau)\|^2 d\tau=0 
\eeq
then 
\beq\label{wlim20}
  \lim_{t\to\infty}\int_U |\nabla p(x,t)|^{2-\beta_2} dx =0.
\eeq
\end{corollary}
\begin{proof}

Combining \eqref{nodelta} and the limit superior of \eqref{wteq8}, we have
\begin{multline*}
\limsup_{t\to\infty} \int_U |\nabla p(x,t)|^{2-\beta_2}dx \\
\le C\Big(1+ \limsup_{t\to\infty} f(t)^\frac{2}{2-\beta_2}+\limsup_{t\to\infty}\int_{t-1}^{t} (\norm{\nabla \Psi}^{2}+\norm{\nabla\Psi_t}^2 +\norm{\Psi_t}^\frac{2-\beta_2}{1-\beta_2} )d\tau\Big),
\end{multline*}
which implies  \eqref{wteq11}.

Let $\delta_1\in (0,1]$. 
On the left-hand side of \eqref{new8}, we neglect the time derivative term and apply \eqref{Hd1} to bound $\mathcal H(x,t)$ from below in terms of $|\nabla p(x,t)|^{2-\beta_2}$.
It yields
\begin{multline}\label{graded}
\delta_1^{\beta_1+\beta_2}\int_U |\nabla p(x,t)|^{2-\beta_2}dx \le C\Big\{  \delta_1^{2+\beta_1} +\delta^{2-\beta_2}  +   \delta^{-(\beta_1+\beta_2)}\norm {\bar p(t-1)}^2 \\
+\delta^{-(\beta_1+\beta_2)}\int_{t-1}^t (\|\nabla \Psi\|^{2-\beta_2}+ \norm{\nabla\Psi}^2+\delta^{-\frac{\beta_1+\beta_2}{1-\beta_2}}\norm{\Psi_t}^{\frac{2-\beta_2}{1-\beta_2}})d\tau\\
+\varep+\int_{t-1}^t(\varep^{-1}\norm{\nabla\Psi_t}^2+ \norm{\Psi_t}^2) d\tau \Big\}.
\end{multline}
 
Under condition \eqref{small30}, we have from Theorem \ref{pasmall} that $\lim_{t\to\infty}\|\bar p(t-1)\|=0$. 
Passing $t\to\infty$ in \eqref{graded} gives
\begin{align*}
 \limsup_{t\to\infty} \int_U |\nabla p(x,t)|^{2-\beta_2}dx 
 &\le C \delta_1^{-(\beta_1+\beta_2)}(  \delta_1^{2+\beta_1} +   \delta^{2-\beta_2}+\varep)\\
 &=C \delta_1^{2-\beta_2} +  C\delta_1^{-(\beta_1+\beta_2)} ( \delta^{2-\beta_2}+\varep).
\end{align*}
Letting $\varep\to 0$, $\delta\to 0$, and then $\delta_1\to 0$, we obtain \eqref{wlim20}.
\end{proof}


\begin{remark} 
(a) Comparing with \eqref{Hest2}, the inequality \eqref{wteq9}  explicitly shows the independence on the initial norm 
$\|\nabla p(0)\|_{L^{2-\beta_2}}$. Also, the term $\int_{t-1}^{t}\norm{\nabla\Psi_t(\tau)}^2 d\tau$ explicitly shows that the dependence on the second derivative $\nabla \Psi_t$ of the boundary data is not accumulative in time on the whole interval $(0,t)$.

(b) Since 
$$\limsup_{t\to\infty} \int_{t-1}^{t}\norm{\nabla\Psi_t(\tau)}^2 d\tau \le \limsup_{t\to\infty} \norm{\nabla\Psi_t(t)}^2,$$
the results \eqref{wteq11}--\eqref{wlim20}  in Corollary \ref{cor310} improve \eqref{limsupG}  in Theorem~\ref{grad-est} and \eqref{small2}--\eqref{wlim2} in Theorem \ref{gradasmall}.
\end{remark}

\section{Continuous dependence on the initial and boundary data}\label{dependsec}

In this section, we establish the continuous dependence of the solution of problem \eqref{IBVP} on the initial and boundary data.
We consider $K(\xi)=\bar K(\xi)$, $K_I(\xi)$, $\hat{K}(\xi)$, $K_M(\xi)$.

Let $p_1(x,t)$ and $p_2(x,t)$ be two solutions of \eqref{IBVP} with boundary data $\psi_1(x,t)$ and $\psi_2(x,t)$, respectively. For $i=1,2$, let $\Psi_i(x,t)$ be an extension of $\psi_i(x,t)$ to $\bar U\times [0,\infty)$, and  $\bar p_i=p_i-\Psi_i$ .
Denote 
$$\Phi=\Psi_1-\Psi_2\quad\text{and}\quad \bar P=\bar p_1-\bar p_2=p_1-p_2-\Phi.$$
 Then 
\begin{align}\label{Wequl}
\frac{\partial \bar P}{\partial t}&= \nabla \cdot (K(|\nabla p_1|)\nabla p_1)-\nabla \cdot (K(|\nabla p_2|)\nabla p_2)-\Phi_t\quad \text{on }U\times(0,\infty),\\
\bar P&=0 \quad \text{on }\Gamma\times(0,\infty).\notag
\end{align}

Let
\beq\label{tilU}
\Lambda(t)=1+\|\nabla p_1( t)\|_{L^{2-\beta_2}}+ \|\nabla p_2(t)\|_{L^{2-\beta_2}}.
\eeq

For the difference between two boundary data, we define
\beqs
D(t)=\norm{\Phi_t(t)}+\norm{\nabla \Phi(t)}_{L^{2-\beta_2}}+\norm{\nabla \Phi(t)}_{L^{2+\beta_1}}^{2+\beta_1}.
\eeqs

First, we obtain the estimates for $\|\bar P(t)\|$ in terms of $D(t)$   and individual solutions $p_1(x,t)$, $p_2(x,t)$.

\begin{proposition}  \label{prop61}
If $t\ge 0$ then
\beq\label{one3}
\| \bar P(t)\|^2 \le \| \bar P(0)\|^2+C\Big\{ Env \Big[ \Lambda(t)^{\beta_1+\beta_2} (\Lambda(t)^{1-\beta_2}+\| \bar p_1 (t) \|+\| \bar p_2 (t) \|) D(t) \Big] \Big\}^\frac{2}{2+\beta_1}.
\eeq

If $\displaystyle \int_0^\infty \Lambda(t)^{-(\beta_1+\beta_2)}dt =\infty$, then  
\beq\label{lmsp1}
\begin{aligned}
 \limsup_{t\to\infty}\|\bar P(t)\|^2 
 &\le C \limsup_{t\to\infty} \Big\{  \Lambda(t)^{\beta_1+\beta_2} (\Lambda(t)^{1-\beta_2}+\|\bar p_1(t)\| + \|\bar p_2(t)\| )
D(t)\Big\}^\frac{2}{2+\beta_1}.
\end{aligned}
\eeq
\end{proposition}
\begin{proof}
First, we find a differential inequality for $\| \bar P(t)\|^2$. We define 
\beqs
\omega(x,t) = 1+|\nabla p_2(x,t)|+|\nabla p_1(x,t)|.
\eeqs

Multiplying \eqref{Wequl} by $\bar P$, integrating the resulting equation over $U$, and
using integration by parts, we obtain 
\begin{align*}
\frac 12 \frac{d}{dt}\int_U \bar P^2dx&= -\int_U[  K(|\nabla p_1|)\nabla p_1- K(|\nabla p_2|)\nabla p_2]\cdot \nabla \bar Pdx-\int_U \Phi_t\bar Pdx,\\
\intertext{thus,}
\frac 12 \frac{d}{dt}\int_U \bar P^2dx&=- \int_U[  K(|\nabla p_1|)\nabla p_1- K(|\nabla p_2|)\nabla p_2]\cdot \nabla (p_1-p_2)dx\\
&\quad + \int_U[  K(|\nabla p_1|)\nabla p_1- K(|\nabla p_2|)\nabla p_2]\cdot \nabla \Phi dx-\int_U \Phi_t\bar Pdx.
\end{align*}

Using the monotonicity in Lemma \ref{lemmono} for the first integral on the RHS, and property \eqref{mc2} with $m=1$ for the second integral, we have  
\beq\label{DiffInq4W}
 \begin{aligned}
\frac 12 \frac{d}{dt}\int_U \bar P^2dx
&\le-d_5\int_U  \frac{|\nabla (p_1- p_2)|^{2+\beta_1}}{\omega^{\beta_1+\beta_2}}dx 
+d_3 \int_U(  |\nabla p_1|^{1-\beta_2}+|\nabla p_2|^{1-\beta_2} )|\nabla \Phi |dx\\
&\quad+\int_U |\Phi_t|(|\bar p_1|+|\bar p_2|)dx \eqdef -J_1+J_2+J_3.
\end{aligned}
\eeq

By \eqref{bi1},
\beq\label{J11}
\begin{aligned}
-J_1
\le - C_1\int_U  \frac{|\nabla \bar P|^{2+\beta_1}}{\omega^{\beta_1+\beta_2}}dx  
+ C\int_U  \frac{|\nabla \Phi|^{2+\beta_1}}{\omega^{\beta_1+\beta_2}}dx\\
\le - C_1\int_U  \frac{|\nabla \bar P|^{2+\beta_1}}{\omega^{\beta_1+\beta_2}}dx  
+ C\int_U  |\nabla \Phi|^{2+\beta_1}dx,
\end{aligned}
\eeq
where $C_1=d_5/2^{1+\beta_1}$. 
Using H\"older's inequality for $J_2$ and $J_3$ terms in \eqref{DiffInq4W}, we find that 
\beq\label{J22}
\begin{split}
J_2&\le C\Big(\norm{\nabla p_1}_{L^{2-\beta_2}}^{1-\beta_2}+ \norm{\nabla p_2}_{L^{2-\beta_2}}^{1-\beta_2}\Big)\norm{\nabla \Phi}_{L^{2-\beta_2}}.
\end{split}
\eeq
\beq\label{J33}
J_3 \le C(\|\bar p_{1}\|+\|\bar p_{2}\| )\norm{\Phi_t}.
\eeq

Utilizing estimates \eqref{J11}--\eqref{J33} in \eqref{DiffInq4W}, we obtain 
\beqs
 \frac d{dt}\| \bar P\|^2 \le - 2C_1\int_U  \frac{|\nabla \bar P|^{2+\beta_1}}{\omega^{\beta_1+\beta_2}}dx  +C\Big(\norm{\nabla \Phi}_{L^{2+\beta_1}}^{2+\beta_1}+ \sum_{i=1,2}\|\nabla p_i\|_{L^{2-\beta_2}}^{1-\beta_2} \|\nabla \Phi\|_{L^{2-\beta_2}}  + \sum_{i=1,2} \| \bar p_i  \|\norm{\Phi_t}\Big).
\eeqs
Thus,
\beq\label{DW}
\begin{aligned}
 \frac d{dt}\| \bar P(t)\|^2 
&\le  - 2C_1\int_U  \frac{|\nabla \bar P|^{2+\beta_1}}{\omega^{\beta_1+\beta_2}}dx   +C( \Lambda(t)^{1-\beta_2} +\|\bar p_1(t)\|+\|\bar p_2(t)\|)D(t).
\end{aligned}
\eeq

Applying H\"{o}lder's inequality with powers $\frac{2+\beta_1}{2-\beta_2}$ and $\frac{2+\beta_1}{\beta_1+\beta_2}$, we have
\beqs
\int_U |\nabla \bar P|^{2-\beta_2} dx
=\int_U \frac{|\nabla \bar P|^{2-\beta_2}}{\omega^\frac{(\beta_1+\beta_2)(2-\beta_2)}{2+\beta_1}} \cdot \omega^\frac{(\beta_1+\beta_2)(2-\beta_2)}{2+\beta_1} dx \le  \Big(\int_U \frac{|\nabla \bar P|^{2+\beta_1}}{\omega^{\beta_1+\beta_2}} dx\Big)^\frac{2-\beta_2}{2+\beta_1}\Big(\int_U \omega^{2-\beta_2} dx\Big)^\frac{\beta_1+\beta_2}{2+\beta_1}.
\eeqs
This implies
\beq\label{em0}
\Lambda(t)^{\beta_1+\beta_2}  \int_U \frac{|\nabla \bar P|^{2+\beta_1}}{\omega^{\beta_1+\beta_2}} dx\ge  \Big(\int_U |\nabla \bar P|^{2-\beta_2} dx\Big)^\frac{2+\beta_1}{2-\beta_2}.
\eeq
Combining \eqref{em0}  with  Poincar\'e-Sobolev inequality \eqref{PSi} yields
\beq\label{embL2}
 \int_U \frac{|\nabla \bar P|^{2+\beta_1}}{\omega^{\beta_1+\beta_2}} dx 
 \ge C_{\rm PS}^{-(2+\beta_1)} \Lambda(t)^{-(\beta_1+\beta_2)} \Big(\int_U \bar P^2 dx \Big)^\frac{2+\beta_1}2.
\eeq

Using  \eqref{embL2} to estimate the first integral on the RHS of \eqref{DW}, we obtain for all $t>0$ that
\beq\label{66}
 \frac d{dt}\| \bar P(t)\|^2 
\le -C_2 \Lambda(t)^{-(\beta_1+\beta_2)} 
\|\bar P(t)\|^{2+\beta_1}  +C( \Lambda(t)^{1-\beta_2}+\|\bar p_1(t)\|+\|\bar p_2(t)\|)D(t),
\eeq
where $C_2=2C_1 C_{\rm PS}^{-(2+\beta_1)}$.
Denote $y(t)=\| \bar P(t)\|^2$ and rewrite \eqref{66} as
\beq\label{6one}
 \frac {dy}{dt} 
\le -C_2 \Lambda(t)^{-(\beta_1+\beta_2)} 
y(t)^\frac{2+\beta_1}{2}  +C( \Lambda(t)^{1-\beta_2}+\|\bar p_1(t)\|+\|\bar p_2(t)\|)D(t).
\eeq

Applying \eqref{ubode} in  Lemma \ref{ODE2} to \eqref{6one} proves \eqref{one3}.
Similarly, applying \eqref{ulode} in Lemma \ref{ODE2} to \eqref{6one}, we obtain \eqref{lmsp1}.
\end{proof}

Next, we combine Proposition \ref{prop61} with results in section \ref{boundsec} to derive more specific estimates.

According to \eqref{west}, 
\beq\label{XvY}
 1+\|\bar p_1\|+\|\bar p_2\|\le \mathcal{X}(t), 
\eeq
where
\beqs
\mathcal{X}(t)= 1+ \sum_{i=1,2} \Big\{\|\bar p_i(0)\|+ \Big(Env f[\Psi_i](t)\Big)^\frac{1}{2-\beta_2}\Big\}.
\eeqs

By \eqref{Hest2} and \eqref{wteq9},
\beq\label{U2}
\Lambda(t)\le  C\hat {\mathcal Y}(t)^{\frac 1{2-\beta_2}},
\eeq
where
\beqs
\hat {\mathcal Y}(t)=
\begin{cases} 
\displaystyle 1+  \sum_{i=1,2}\Big(\|\bar p_i(0)\|^2 +e^{-\frac t 2}\|\nabla  p_i(0)\|_{L^{2-\beta_2}}^{2-\beta_2} \\
\quad  +  (Env f[\Psi_i](t))^\frac{2}{2-\beta_2}+ \int_0^t e^{-\frac 12(t-\tau)} \| \nabla \Psi_{i,t} (\tau)\|^2 d\tau\Big) &\text{ if }  0\le t<1,  \\
\displaystyle 1+\sum_{i=1,2} \Big(\|\bar p_i(0)\|^2+ (Env f[\Psi_i](t))^\frac{2}{2-\beta_2}+\int_{t-1}^{t}\norm{\nabla\Psi_{i,t}(\tau)}^2 d\tau\Big) &\text{ if }  t\ge 1.
\end{cases}
\eeqs

To simplify  expressions of our estimates, we set 
\beq\label{Y0def}
\mathcal Y_0=  1+ \sum_{i=1,2}\Big(\|\bar p_i(0)\|^2 +\|\nabla  p_i(0)\|_{L^{2-\beta_2}}^{2-\beta_2} \Big),
\eeq
and define the function
\beq\label{deftilY}
\widetilde{\mathcal Y}(t)= \mathcal Y_0+\sum_{i=1,2} (Env f[\Psi_i](t))^\frac{2}{2-\beta_2}+
\begin{cases} 
  \int_0^t e^{-\frac 12(t-\tau)} \sum_{i=1,2}\| \nabla \Psi_{i,t} (\tau)\|^2 d\tau&\text{ if }  0\le t<1,  \\
\int_{t-1}^{t}\sum_{i=1,2}\norm{\nabla\Psi_{i,t}(\tau)}^2 d\tau &\text{ if }  t\ge 1.
\end{cases}
\eeq

Then
$\hat {\mathcal Y}(t)\le \widetilde{\mathcal Y}(t)$ and $\mathcal X(t)\le \widetilde{\mathcal Y}^{1/2}(t)$.
These properties and \eqref{XvY}, \eqref{U2} imply  
\beq\label{U3}
\Lambda(t)\le  C\widetilde {\mathcal Y}(t)^{\frac 1{2-\beta_2}},
\eeq
\beq\label{YY} 
 \Lambda(t)^{1-\beta_2}+\|\bar p_1(t)\|+\|\bar p_2(t)\|\le C(\widetilde{\mathcal Y}(t)^\frac{1-\beta_2}{2-\beta_2}+\widetilde{\mathcal Y}(t)^{1/2})\le C\widetilde{\mathcal Y}(t)^{1/2}. \eeq
Above, we used the fact $1/2>(1-\beta_2)/(2-\beta_2)$ and  $\widetilde{\mathcal Y}(t) \ge 1$ . 

For asymptotic estimates, we will use the following numbers
\beq\label{tilAKdef}
\begin{aligned}
\widetilde{\mathcal{A}}&=\Big(\sum_{i=1,2}\limsup_{t\to\infty} f[\Psi_i](t)\Big)^\frac1{2-\beta_2},\quad
\widetilde{\mathcal K}={\widetilde{\mathcal{A}}}^2
+\sum_{i=1,2}\limsup_{t\to\infty}\int_{t-1}^{t}\norm{\nabla\Psi_{i,t}(\tau)}^2d\tau,\\
\mathcal D&=\limsup_{t\to\infty}D(t).
\end{aligned}
\eeq

Now we can estimate the $L^2$-norm of $\bar P(t)$ utterly in terms of the initial and boundary data. 


\begin{theorem}\label{theo62}
For $t\ge 0$,
\beq\label{neww}
\|\bar P(t)\|^2 \le \|\bar P(0)\|^2+C\Big\{ Env \Big[ \widetilde{\mathcal Y}(t)^{\frac {\beta_1+\beta_2}{2-\beta_2}+\frac12} D(t)\Big] \Big\}^\frac{2}{2+\beta_1}.
\eeq

If $\widetilde{\mathcal K}<\infty$  then
\beq\label{6one3}
 \limsup_{t\to\infty}\|\bar P(t)\|^2 \le 
C \Big\{(1 +\widetilde{\mathcal K})^{\frac {\beta_1+\beta_2}{2-\beta_2}+\frac12} \mathcal D\Big\}^\frac{2}{2+\beta_1}.
\eeq 
\end{theorem}
\begin{proof}
It follows from \eqref{one3}, \eqref{U3} and  \eqref{YY} that
\beqs
\|\bar P(t)\|^2 \le \|\bar P(0)\|^2+C\Big\{ Env \Big[ \widetilde{\mathcal Y}(t)^{\frac{\beta_1+\beta_2} {2-\beta_2}}\widetilde{\mathcal Y}(t)^\frac12 D(t) \Big] \Big\}^\frac{2}{2+\beta_1}
\eeqs
which implies  \eqref{neww}.

We have from limit estimates  \eqref{wteq11} and  \eqref{nodelta} that
\beq\label{limsup}
\limsup_{t\to\infty} \Lambda(t) \le C(1 +\widetilde{\mathcal K})^\frac1{2-\beta_2}, \quad
\limsup_{t\to\infty}(1+\|\bar p_1(t)\| +\|\bar p_2(t)\|)\le  C(1+\widetilde{\mathcal{A}})\le  C(1+{\widetilde{\mathcal K}})^\frac{1}{2}.
\eeq
Combining \eqref{limsup} with \eqref{lmsp1} we obtain
\beq\label{Lsup2}
 \limsup_{t\to\infty}\|\bar P(t)\|^2
\le 
C \Big\{(1 +\widetilde{\mathcal K})^\frac {\beta_1+\beta_2}{2-\beta_2}
\Big[(1+\widetilde{\mathcal K})^\frac{1}{2}+(1 +\widetilde{\mathcal K})^\frac {1-\beta_2}{2-\beta_2}\Big] \mathcal D\Big\}^\frac{2}{2+\beta_1}.
\eeq 
Note that $1/2>(1-\beta_2)/(2-\beta_2)$, 
then  \eqref{6one3} follows \eqref{Lsup2}.
\end{proof}

\section{Structural stability}\label{strucsec}
In this section, we consider the case $K(\xi)=K_I(\xi,\vec a)$ in \eqref{KIxa}, and study the dependence of the solutions to IBVP  \eqref{IBVP}  on the coefficient vector $\vec a$. 

Let $N\ge 1$ and  the exponent vector $\vec{\alpha} =(-\alpha, 0, \alpha_1, \ldots, \alpha_N)$ be fixed.
Since $\vec a$ satisfies condition \eqref{aicond}, we denote the set of admissible $\vec a$ by $S$, that is,
\beqs
S=\{ \vec{a} =(a_{-1}, a_0, \ldots, a_N): a_{-1}, a_N>0, a_0, a_1, \ldots, a_{N-1}\ge 0\}.
\eeqs

The following ``perturbed monotonicity'' is important for our structural stability in this section; it plays the same role as the monotonicity (Lemma \ref{lemmono}) for the continuous dependence in section \ref{dependsec}.
Below, the notation $\vee$, resp. $\wedge$,  denotes the maximum, resp. minimum, of two numbers or two vectors meaning coordinate-wise.
\begin{lemma}[Perturbed Monotonicity]\label{lempm}
Let  $K_I(\xi,\vec a)$ be defined as in  \eqref{KIxa}. 
For any coefficient vectors  $\vec a^{(1)}$, $\vec a^{(2)}\in S$,  and  any  $y,y'\in \mathbb R^n$, one has 
\begin{multline}\label{monoper}
\big(K_I(|y'|,\vec a^{(1)})y'-K_I(|y|,\vec a^{(2)})y\big)\cdot (y'-y) \ge \frac{d_6  |y-y'|^{2+\beta_1}}{(1+|y|+|y'|)^{\beta_1+\beta_2}}\\
 -d_7 K\big(|y|\vee |y'|,\vec a^{(1)}\wedge \vec a^{(2)}\big)  \ (|y|\vee |y'|) \  |\vec{a}^{(1)}-\vec{a}^{(2)}|\ |y-y'|,
\end{multline}
where $d_6=d_6(\vec a^{(1)},\vec a^{(2)})$ and $d_7=d_7(\vec a^{(1)},\vec a^{(2)})$ are positive constants defined by
\begin{align*}
d_6&=\frac{ 1-\beta_2}{(\beta_1+1)\Big [2(N+2)\max\big\{1,a_i^{(j)}:i=-1,0,\ldots,N,\ j=1,2\big\}\Big]^{\beta_1+1}} , 
\\
d_7&=\frac{N+1}{(1-\alpha)\min\big\{a_{-1}^{(1)},a_{-1}^{(2)},a_N^{(1)},a_N^{(2)}\big\}}. 
\end{align*}
\end{lemma}
\begin{proof} 
Let  $\vec a^{(1)}$, $\vec a^{(2)}\in S$  and  $y,y'\in \mathbb R^n$.
Same as in Lemma \ref{lemmono}, it suffices to consider the case when the line segment $[y,y']$ does not contain the origin.
For $t \in[0,1]$, let 
$$\gamma(t)=ty+(1-t)y',\quad \vec{b}(t)=(b_{-1}(t),b_0(t),\ldots,b_N(t))\eqdef t\vec{a}^{(1)}+(1-t)\vec{a}^{(2)},$$
and define 
 $$z(t)=K(|\gamma(t)|,\vec{b}(t))\, \gamma(t)\cdot (y-y').$$
 
We have 
\beqs
I
\eqdef [K(|y|,\vec{a}^{(1)})y-K(|y'|,\vec{a}^{(2)})y']\cdot (y'-y)
=z(1)-z(0)=\int_0^1 z'(t) dt.
\eeqs

In calculations below, we use the following short-hand notation for partial derivatives 
$$X_s=\partial X/\partial s,\quad X_\xi=\partial X/\partial \xi,\quad X_{a_i}=\partial X/\partial a_i, \text{ and } X_{\vec a}=\partial X/\partial \vec a.$$

Elementary calculations give
\beq\label{Ifm}
I=\int_0^1 h_1(t)dt + \int_0^1 h_2(t) dt\eqdef I_1+I_2,
\eeq
where
\begin{align*}
h_1(t)&=K(|\gamma(t)|,\vec{b}(t))|y-y'|^2
 +K_\xi(|\gamma(t)|,\vec{b}(t))\frac{|\gamma(t)\cdot(y-y')|^2}{|\gamma(t)|},\\
h_2(t)&= K_{\vec{a}}(|\gamma(t)|,\vec{b}(t))(\vec{a}^{(1)}-\vec{a}^{(2)})\gamma(t)\cdot (y-y') .
\end{align*}

 $\bullet$ \emph{Estimation of $I_1$.} By Lemma \ref{LemKp},
\begin{align*}
K_\xi(|\gamma(t)|,\vec{b}(t))\ge -\beta_2 \frac{K(|\gamma(t)|,\vec{b}(t))}{|\gamma(t)|}.
\end{align*}

Same as the proof of \eqref{hp},
\beq\label{ks}
h_1(t)
  \ge  \frac{d_2(t) (1-\beta_2)|y'-y|^2}{(1+|y|+|y'|)^{\beta_1+\beta_2}} |\gamma(t)|^{\beta_1},
\eeq
where
\beqs
d_2(t)=\frac1{(\max\{1,\sum_{i=-1}^N \vec{b}_i(t)\})^{1+\beta_1}} .
\eeqs
We can estimate $d_2(\vec b(t))\ge d_2^*$ for all $t\in[0,1]$, where
\beqs
d_2^*=\frac1{((N+2)\max\{1,a_i^{(j)}:i=-1,0,\ldots,N,\ j=1,2\})^{1+\beta_1}} .
\eeqs
 Hence, it follows \eqref{ks} that 
\beq\label{int-h1}
\int_0^1 h_1(t) dt
  \ge  \frac{d_2^* (1-\beta_2)|y'-y|^2}{(1+|y|+|y'|)^{\beta_1+\beta_2}} \int_0^1 |\gamma(t)|^{\beta_1} dt.
\eeq
Same calculations in \eqref{ig1} and \eqref{ig2} of Lemma \ref{lemmono} show that
\beq\label{int-gam}
\int_0^1 |\gamma(t)|^{\beta_1} dt \ge  \frac{ |y'-y|^{\beta_1}}{2^{\beta_1+1}(\beta_1+1) } .
\eeq
It follows from \eqref{int-h1} and \eqref{int-gam} that
\beq\label{ksame}
I_1= \int_0^1 h_1(t)dt
 \ge \frac{d_6|y-y'|^{2+\beta_1}}{(1+|y|+|y'|)^{\beta_1+\beta_2}} .
\eeq

$\bullet$ \emph{Estimation of $I_2$.} We find the partial derivative of $K(\xi,\vec{a})$ in $\vec{a}$.
In calculations below, we denote, for convenience,  $\alpha_{-1}=-\alpha$.

For $i=-1,0,1,\ldots,N$,  taking the partial derivative in $a_i$  of the identity $K(\xi,\vec{a})=1/g(s(\xi,\vec{a}),\vec{a})$, we find that
\begin{align*}
 K_{a_i}(\xi,\vec{a})=-\frac{g_{a_i}+g_s\cdot  s_{a_i}}{g^2}=-K(\xi,\vec{a})\frac{g_{a_i}+g_s\cdot  s_{a_i}}{g}.
\end{align*}
Similarly, from  $sg(s,\vec a)=\xi $, we have for $i=-1,0,1,\dots, N$,
\begin{align*}
s_{a_i}\cdot g+s\cdot ( g_{a_i}+g_s\cdot  s_{a_i} )=0,
\end{align*}
which implies 
\beqs
s_{a_i}=\frac{-s\cdot g_{a_i}}{g+s\cdot g_s}.
\eeqs
Then we obtain
\beq\label{Kai}
K_{a_i}(\xi,\vec{a})=-K(\xi,\vec{a})\frac{g_{a_i}+g_s\cdot  \frac{-s\cdot g_{a_i}}{g+s\cdot g_s}}{g}
=-K(\xi,\vec{a})\frac{ g_{a_i}}{g+s\cdot g_s}=-K(\xi,\vec{a})\frac{ s^{\alpha_i}}{g+s\cdot g_s},
\eeq
and consequently,
\beqs
 \sum_{i=-1}^N |K_{a_i}(\xi,\vec{a})|
\le K(\xi,\vec{a}) \frac{ s^{-\alpha}+1+s^{\alpha_1}+\cdots+s^{\alpha_N} }{ (1-\alpha)a_{-1}s^{-\alpha}+a_0+(1+\alpha_1)a_1s^{\alpha_1}+\cdots+(1+\alpha_N)a_Ns^{\alpha_N} }.
\eeqs
Using \eqref{bi3}, we have
\beqs
1,s^{\alpha_1},\ldots, s^{\alpha_{N-1}}\le s^{-\alpha}+s^{\alpha_N}.
\eeqs
Hence,
\beqs
 \sum_{i=-1}^N |K_{a_i}(\xi,\vec{a})|
\le K(\xi,\vec{a}) \frac{(N+1)(s^{-\alpha}+s^{\alpha_N}) }{ (1-\alpha)a_{-1}s^{-\alpha}+(1+\alpha_N)a_N s^{\alpha_N} }
\le d (\vec{a})K(\xi,\vec{a}).
\eeqs
where
\beqs
d(\vec{a})=\frac{N+1}{\min\{(1-\alpha)a_{-1},(1+\alpha_N)a_N\}}.
\eeqs
Thus,
\beq\label{Ka}
|K_{\vec a}(\xi,\vec{a})| \le d(\vec{a}) K(\xi,\vec{a}).
\eeq

Now, by estimate \eqref{Ka}
\beq\label{ih2}
\begin{aligned}
| h_2(t)|
&\le |K_{\vec{a}}(|\gamma(t)|,\vec{b}(t))|\cdot |\vec{a}^{(1)}-\vec{a}^{(2)}|\cdot  | \gamma(t) |\cdot |y-y'|\\
&\le d(\vec{b}(t)) K(|\gamma(t)|,\vec{b}(t))| \gamma(t) |  |\vec{a}^{(1)}-\vec{a}^{(2)}|  \cdot |y-y'|.
\end{aligned}
\eeq

Since $a_i^{(j)}$ is positive for $i=-1,N$ and $j=1,2$, the number $d(\vec{b}(t))$, for all $t\in[0,1]$, can be bounded  by
\beqs 
d(\vec{b}(t)) \le \frac{N+1}{\min\{(1-\alpha)a_{-1}^{(1)},(1-\alpha)a_{-1}^{(2)},(1+\alpha_N)a_N^{(1)},(1+\alpha_N)a_N^{(2)}\}}\le d_7.
\eeqs

Applying Corollary \ref{incor} to $m=1\ge\beta_2$ gives that the function  $\xi K(\xi,\vec{b}(t))$ is increasing in $\xi$. This, together with \eqref{ih2}  and  the fact $|\gamma(t)|\le |y|\vee |y'|$, yields
\beqs
| h_2(t)|
\le d_7 K(|y|\vee |y'|,\vec{b}(t))  (|y|\vee |y'|)   |\vec{a}^{(1)}-\vec{a}^{(2)}|  \cdot |y-y'|.
\eeqs

Note from \eqref{Kai} that $K(\xi,\vec a)$ is decreasing in each $a_i$, hence 
\beq
K(\xi,\vec{b}(t))\le K(\xi,\vec a^{(1)}\wedge \vec a^{(2)}).
\eeq
Therefore,
\beqs
| h_2(t)|
\le d_7 K(|y|\vee |y'|,\vec a^{(1)}\wedge \vec a^{(2)}) (|y|\vee |y'|)   |\vec{a}^{(1)}-\vec{a}^{(2)}| |y-y'|,
\eeqs
and consequently,
\beq \label{Ke}
I_2\ge - \int_0^1 | h_2(t)|dt 
\ge - d_7 K(|y|\vee |y'|,\vec a^{(1)}\wedge \vec a^{(2)}) (|y|\vee |y'|)   |\vec{a}^{(1)}-\vec{a}^{(2)}| |y-y'|.
\eeq

Thus, we obtain \eqref{monoper} by combining \eqref{Ifm}, \eqref{ksame} and \eqref{Ke}.
\end{proof}


Let  $\coefset$ be a compact subset of $S$ and  let the boundary data $\psi(x,t)$ be fixed.

For $i=1,2$, let $\vec a^{(i)}\in \coefset$,  and let $p_i(x,t)$ be the  solution of \eqref{IBVP} with $K=K(\xi,\vec a^{(i)})$. 
Our goal is to estimate $p_1(x,t)-p_2(x,t)$ in terms of $\vec a^{(1)}-\vec a^{(2)}$. 

We will use the results in  section \ref{boundsec} for estimates of $p_1$ and $p_2$. 
Examining constants $d_2$, $d_3$, $d_4$ in section \ref{presec}, and $d_6$, $d_7$ in Lemma \ref{lempm}, we see that they can be made dependent only on $N$, $\alpha$, $\alpha_N$, $\beta_1$, $\beta_2$ and the following constants
\begin{align*}
\bar c_\coefset&=\max\{a_i: -1\le i\le N, \vec a=(a_i)_{i=-1}^N\in \coefset\}, \\
\underline c_\coefset&=\min\{a_{-1},a_N:\vec a=(a_i)_{i=-1}^N\in \coefset\}. 
\end{align*}

Consequently, the constants $C,C_0,C_1,\dots$ in calculations and bounds in section \ref{boundsec} can be made dependent only on $N$, $\alpha$, $\alpha_N$, $\beta_1$, $\beta_2$, $\bar c_\coefset$, $\underline c_\coefset$ and $C_{\rm PS}$.  
Such dependence will also apply to the generic, positive constant $C$ in this section. 

Let $\Psi$ be the extension of $\psi$ as in section \ref{L2sec}. The calculations in section \ref{dependsec}, when used in this section, will correspond to $\Psi_1=\Psi_2=\Psi$.

Let $P=p_1-p_2$. We have 
\begin{align}\label{ss1}
\frac{\partial P}{\partial t}&= \nabla \cdot \Big(K(|\nabla p_1|,\vec{a}^{(1)})\nabla p_1-K(|\nabla p_2|,\vec{a}^{(2)})\nabla p_2\Big)\quad \text{on }U\times(0,\infty),\\
P&=0\quad \text{on }\Gamma\times(0,\infty).\notag
\end{align}

Let $f(t)$  be defined by \eqref{fdef}, $\Lambda(t)$ by \eqref{tilU}, and $\mathcal Y_0$ by \eqref{Y0def}.
 We define, similar to \eqref{deftilY}, the function
\beqs
 {\mathcal Y}(t)=\mathcal Y_0 + (Env f(t))^\frac{2}{2-\beta_2}+ \begin{cases} 
 \int_0^t \norm{\nabla\Psi_t(\tau)}^2d\tau &\text{if }  0\le t<1,  \\
\int_{t-1}^{t}\norm{\nabla\Psi_t(\tau)}^2 d\tau&\text{if }  t\ge 1,
\end{cases} 
\eeqs
and, similar to \eqref{tilAKdef}, the numbers
\beq\label{AK}
\mathcal{A} =\limsup_{t\to\infty} f(t)^\frac1{2-\beta_2}
\quad\text{and}\quad
 \mathcal K ={\mathcal{A}}^2+\limsup_{t\to\infty}\int_{t-1}^t \norm{ \nabla\Psi_{t}(\tau) }^2d\tau.
\eeq
\begin{theorem} \label{DepCoeff} 
\
\begin{enumerate}
\item[\rm (i)] For $t\ge 0$, one has
\beq\label{ssc2}
\int_U |P(x,t)|^2 dx\le \int_U |P(x,0)|^2dx+C Env\,{ \mathcal Y}(t)^{\frac{2}{2-\beta_2}} |\vec{a}^{(1)}-\vec{a}^{(2)}|^\frac{2}{2+\beta_1} .
\eeq

\item[\rm (ii)] If  $\mathcal K<\infty$ then 
\beq\label{ssc3}
\limsup_{t\to\infty}\int_U |P(x,t)|^2 dx\le  C(1+\mathcal K)^\frac{2}{2-\beta_2} |\vec{a}^{(1)}-\vec{a}^{(2)}|^\frac{2}{2+\beta_1}.
\eeq
\end{enumerate}
\end{theorem}
\begin{proof}
Multiplying equation  \eqref{ss1} by $P$, integrating over $U$, and by integration by parts, we find that
\begin{align*}
\frac 12\frac{d}{dt}\int_U P^2 dx= -\int_U (K(|\nabla p_1|,\vec{a}^{(1)})\nabla p_1-K(|\nabla p_2|,\vec{a}^{(2)})\nabla p_2)\cdot (\nabla p_1-\nabla p_2)dx.
\end{align*}

By the perturbed monotonicity  \eqref{monoper} of $K(\xi,\vec{a})$, we have 
\beq\label{ssc4}
\frac 12\frac{d}{dt}\int_U P^2 dx
\le -d_6J+C  |\vec{a}^{(1)}-\vec{a}^{(2)}|\int_UK(|\nabla p_1|\vee |\nabla p_2|,\vec a^{(1)}\wedge \vec a^{(2)}) (|\nabla p_1|\vee |\nabla p_2|)^2  dx,
\eeq
where
\beqs
J=\int_U\frac{|\nabla P|^{2+\beta_1}}{(1+|\nabla p_1|+|\nabla p_2|)^{\beta_1+\beta_2}} dx.
\eeqs
Using \eqref{mc2} with $m=2$ for the last integral of \eqref{ssc4}, we have 
\beq\label{keyEst}
\frac 12\frac{d}{dt}\int_U P^2 dx
\le -d_6J+C |\vec{a}^{(1)}-\vec{a}^{(2)}|\int_U ( |\nabla p_1|^{2-\beta_2}+ |\nabla p_2|^{2-\beta_2}) dx.
\eeq
In \eqref{keyEst}, estimating $J$ by \eqref{embL2} we have, same as \eqref{6one},
\beq\label{Hu}
\frac 12\frac{d}{dt}\int_U P^2 dx\le - C_1\Big(\int_U P^2 dx\Big)^\frac{2+\beta_1}{2}\Lambda(t)^{-(\beta_1+\beta_2)}+C |\vec{a}^{(1)}-\vec{a}^{(2)}|\Lambda(t)^{2-\beta_2},
\eeq
where $C_1=2d_6 C_{\rm PS}^{-(2+\beta_1)}$.

(i) Same as \eqref{U3},  there is $C_2>0$ such that 
$\Lambda(t) \le C_2 {\mathcal Y}(t)^{\frac{1}{2-\beta_2}}$.
 Hence, we find that  
\beq\label{Hu1}
\frac 12\frac{d}{dt}\int_U P^2 dx\le - C_3\Big(\int_U P^2 dx\Big)^\frac{2+\beta_1}{2}{\mathcal Y}(t)^{-\frac{\beta_1+\beta_2}{2-\beta_2}}+C |\vec{a}^{(1)}-\vec{a}^{(2)}| {\mathcal Y}(t)
\eeq
for some $C_3>0$.
 Applying \eqref{ubode} of Lemma \ref{ODE2} to  differential inequality \eqref{Hu1}, we obtain 
 \beqs
\int_U |P(x,t)|^2 dx\le \int_U |P(x,0)|^2dx+C\Big\{ |\vec{a}^{(1)}-\vec{a}^{(2)}| Env\, { \mathcal Y}(t)^{1+\frac{\beta_1+\beta_2}{2-\beta_2}}\Big\}^\frac{2}{2+\beta_1} ,
\eeqs
 thus,  \eqref{ssc2} follows.

(ii) By the virtue of  \eqref{wteq11} and \eqref{AK},  
\beq\label{ssc5}
\limsup_{t\to\infty} \Lambda(t) \le 1+ \sum_{i=1,2}\limsup_{t\to\infty} \norm{\nabla p_i}_{L^{2-\beta_2}} \le C \big( 1+ \mathcal K\big)^\frac{1}{2-\beta_2}<\infty.
\eeq
Thus,
$$\displaystyle \int_0^\infty\Lambda(t)^{-(\beta_1+\beta_2)}dt =\infty.$$
Applying Lemma \ref{ODE2} to \eqref{Hu}, for $\theta=\frac{2+\beta_1}2$, we have 
\beq\label{baseineq3}
\begin{aligned}
\limsup_{t\to\infty}\int_U P^2 dx\le  C\Big[ |\vec{a}^{(1)}-\vec{a}^{(2)}|\limsup_{t\to\infty} \Lambda(t)^{2+\beta_1} \Big]^\frac{2}{2+\beta_1}= C |\vec{a}^{(1)}-\vec{a}^{(2)}|^\frac{2}{2+\beta_1} \limsup_{t\to\infty} \Lambda^2(t).
\end{aligned}
\eeq
Therefore,  we obtain  \eqref{ssc3} from \eqref{baseineq3} and \eqref{ssc5}.
\end{proof}

\myclearpage

\def\cprime{$'$}


\begin{thebibliography}{10}

\bibitem{ABHI1}
E.~Aulisa, L.~Bloshanskaya, L.~Hoang, and A.~Ibragimov.
\newblock {Analysis of generalized {F}orchheimer flows of compressible fluids
  in porous media}.
\newblock {\em J. Math. Phys.}, 50(10):103102:44pp, 2009.

\bibitem{BearBook}
J.~Bear.
\newblock {\em {Dynamics of Fluids in Porous Media}}.
\newblock American Elsevier Pub. Co., New York, 1972.

\bibitem{CH2}
E.~Celik and L.~Hoang.
\newblock Maximum estimates for generalized {F}orchheimer flows in
  heterogeneous porous media.
\newblock 2015.
\newblock submitted, preprint http://arxiv.org/abs/1510.09000.

\bibitem{CH1}
E.~Celik and L.~Hoang.
\newblock Generalized {F}orchheimer flows in heterogeneous porous media.
\newblock {\em Nonlinearity}, 29(3):1124--1155, 2016.

\bibitem{CHK1}
E.~Celik, L.~Hoang, and T.~Kieu.
\newblock Generalized {F}orchheimer flows of isentropic gases.
\newblock 2015.
\newblock submitted, preprint http://arxiv.org/abs/1504.00742.

\bibitem{CHK2}
E.~Celik, L.~Hoang, and T.~Kieu.
\newblock Doubly nonlinear parabolic equations for a general class of
  {F}orchheimer gas flows in porous media.
\newblock 2016.
\newblock submitted, preprint http://arxiv.org/abs/1601.00703.

\bibitem{DiDegenerateBook}
E.~DiBenedetto.
\newblock {\em {Degenerate parabolic equations}}.
\newblock {Universitext}. Springer-Verlag, New York, 1993.

\bibitem{Dudgeon85}
C.~Dudgeon.
\newblock {\em Non-{D}arcy flow of groundwater. {P}art 1. {T}heoretical,
  experimental and numerical studies. Report No. 162}.
\newblock Water research laboratory, {T}he {U}niversity of {N}ew {S}outh
  {W}ales, 1985.

\bibitem{Forchh1901}
P.~Forchheimer.
\newblock {Wasserbewegung durch Boden}.
\newblock {\em Zeit. Ver. Deut. Ing.}, 45:1781--1788, 1901.

\bibitem{ForchheimerBook}
P.~Forchheimer.
\newblock {\em {{H}ydraulik}}.
\newblock Number Leipzig, Berlin, B. G. Teubner. 1930.
\newblock 3rd edition.

\bibitem{HI1}
L.~Hoang and A.~Ibragimov.
\newblock {Structural stability of generalized {F}orchheimer equations for
  compressible fluids in porous media}.
\newblock {\em Nonlinearity}, 24(1):1--41, 2011.

\bibitem{HI2}
L.~Hoang and A.~Ibragimov.
\newblock Qualitative study of generalized {F}orchheimer flows with the flux
  boundary condition.
\newblock {\em Adv. Diff. Eq.}, 17(5--6):511--556, 2012.

\bibitem{HIKS1}
L.~Hoang, A.~Ibragimov, T.~Kieu, and Z.~Sobol.
\newblock Stability of solutions to generalized {F}orchheimer equations of any
  degree.
\newblock {\em J. Math. Sci.}, 210(4):476--544, 2015.

\bibitem{HK1}
L.~Hoang and T.~Kieu.
\newblock Interior estimates for generalized {F}orchheimer flows of slightly
  compressible fluids.
\newblock 2014.
\newblock submitted, preprint http://arxiv.org/abs/1404.6517.

\bibitem{HK2}
L.~Hoang and T.~Kieu.
\newblock Global estimates for generalized {F}orchheimer flows of slightly
  compressible fluids.
\newblock {\em Journal d'Analyse Mathematique}, 2015.
\newblock accepted.

\bibitem{HIK1}
L.~T. Hoang, A.~Ibragimov, and T.~T. Kieu.
\newblock One-dimensional two-phase generalized {F}orchheimer flows of
  incompressible fluids.
\newblock {\em J. Math. Anal. Appl.}, 401(2):921--938, 2013.

\bibitem{HIK2}
L.~T. Hoang, A.~Ibragimov, and T.~T. Kieu.
\newblock A family of steady two-phase generalized {F}orchheimer flows and
  their linear stability analysis.
\newblock {\em J. Math. Phys.}, 55(12):123101:32pp, 2014.

\bibitem{HKP1}
L.~T. Hoang, T.~T. Kieu, and T.~V. Phan.
\newblock Properties of generalized {F}orchheimer flows in porous media.
\newblock {\em J. Math. Sci.}, 202(2):259--332, 2014.

\bibitem{JerisonKenig1995}
D.~Jerison and C.~E. Kenig.
\newblock The inhomogeneous {D}irichlet problem in {L}ipschitz domains.
\newblock {\em J. Funct. Anal.}, 130(1):161--219, 1995.

\bibitem{Muskatbook}
M.~Muskat.
\newblock {\em {The flow of homogeneous fluids through porous media}}.
\newblock McGraw-Hill Book Company, inc., 1937.

\bibitem{NieldBook}
D.~A. Nield and A.~Bejan.
\newblock {\em {Convection in porous media}}.
\newblock Springer-Verlag, New York, fourth edition, 2013.

\bibitem{SSHI2016}
F.~Siddiqui, M.~Y. Soliman, W.~House, and A.~Ibragimov.
\newblock Pre-{D}arcy flow revisited under experimental investigation.
\newblock {\em Journal of {A}nalytical {S}cience and {T}echnology}, 7(2), 2016.

\bibitem{SoniIslamBasak78}
J.~Soni, N.~Islam, and P.~Basak.
\newblock An experimental evaluation of non-{D}arcian flow in porous media.
\newblock {\em Journal of {H}ydrology}, 38(3-4):231--241, 1978.

\bibitem{StraughanBook}
B.~Straughan.
\newblock {\em {Stability and wave motion in porous media}}, volume 165 of {\em
  {Applied Mathematical Sciences}}.
\newblock Springer, New York, 2008.

\bibitem{VazquezPorousBook}
J.~L. V{\'a}zquez.
\newblock {\em {The porous medium equation}}.
\newblock {Oxford Mathematical Monographs}. The Clarendon Press Oxford
  University Press, Oxford, 2007.
\newblock Mathematical theory.

\end{thebibliography}
%
\end{document}